\documentclass{article}

\usepackage{amsmath,hyperref,amsthm,algorithm2e,amsfonts,braket,xcolor,graphicx}
\usepackage[margin=1.2in]{geometry}
\hypersetup{linkbordercolor={1 .8 .8},citebordercolor={.7 1 .7}}
\newtheorem{prp}{Proposition}
\newtheorem{lem}{Lemma}
\newtheorem{thm}{Theorem}
\newtheorem{cly}{Corollary}

\theoremstyle{definition}

\newtheorem{dfn}{Definition}
\newcommand{\Li}{\text{Li}}
\newcommand{\ty}{\text{ty}}
\newcommand{\ve}{\mathbf{e}}
\newcommand{\qHyp}{q\text{Hyp}}
\newcommand{\qGeom}{q\text{Geom}}
\newcommand{\RSK}{\text{RSK}}

\newcommand{\pint}{\mathbb N}
\newcommand{\real}{\mathbb R}
\newcommand{\prob}{\mathbb P}
\newcommand{\expe}{\mathbb E}
\newcommand{\ind}{\mathbb I}
\newcommand{\vmat}[5]{
  \begin{pmatrix}
    #1 & & #2 \\ #3 & \underset{#5}{\text{---}} & #4
    \end{pmatrix}
  }
\begin{document}
\title{A $q$-Robinson-Schensted-Knuth algorithm and a $q$-polymer}
\author{Yuchen Pei}
\newcommand{\Addresses}{{
  \bigskip
  \footnotesize

  Yuchen Pei, \textsc{Center of Mathematical Sciences and Applications, Harvard University}\par\nopagebreak
  \textit{E-mail address}: \texttt{me@ypei.me}

}}

\maketitle
\abstract{
  In \cite{matveev-petrov15} a $q$-deformed Robinson-Schensted-Knuth algorithm ($q$RSK) was introduced.
In this article we give reformulations of this algorithm in terms of the Noumi-Yamada description, growth diagrams and local moves. 
We show that the algorithm is symmetric, namely the output tableaux pairs are swapped in a sense of distribution when the input matrix is transposed. 
We also formulate a $q$-polymer model based on the $q$RSK, prove the corresponding Burke property, which we use to show a strong law of large numbers for the partition function given stationary boundary conditions and $q$-geometric weights.
We use the $q$-local moves to define a generalisation of the $q$RSK taking a Young diagram-shape of array as the input.
We write down the joint distribution of partition functions in the space-like direction of the $q$-polymer in $q$-geometric environment, formulate a $q$-version of the multilayer polynuclear growth model ($q$PNG) and write down the joint distribution of the $q$-polymer partition functions at a fixed time.
}

\section{Introduction and main results}\label{s:intro}
The RSK algorithm was introduced in \cite{knuth70} as a generalisation of the Robinson-Schensted (RS) algorithm introduced in \cite{robinson38,schensted61}.
It transforms a matrix to a pair of semi-standard Young tableaux.
For an introduction of the RS(K) algorithms and Young tableaux see e.g. \cite{fulton97,sagan00}.

The gRSK algorithm, as a geometric lifting of the RSK algorithm, that is replacing the max-plus algebra by the usual algebra in its definition, was introduced in \cite{kirillov01}.

There are several equivalent formulations of the (g)RSK algorithms.
The commonest definition of the RSK algorithm is based on inserting a row of the input matrix into a semi-standard Young tableau.
However for the needs of defining gRSK, the insertion was reformulated as a map transforming a tableau row and an input row into a new tableau row and the input row to insert into the next tableau row.
This was introduced in \cite{noumi-yamada04}, and henceforce we call it the Noumi-Yamada description.
It will be the first reformulation of the algorithms in this article, from which we derive all of our main results.

The symmetry properties state that the pair of output tableaux are swapped if the input matrix is transposed.
One way to prove this is to reformulate the RSK algorithms as growth diagrams.
The growth diagram was developed in \cite{fomin86,fomin95}, see also exposition in \cite[Section 5.2]{sagan00}.
It is a rectangular lattice graph whose vertices record the growth of the shape of the output tableaux, and can be generated recursively by the local growth rule.

Much of the attention the (g)RSK algorithms receive these days come from the relation to the directed last passage percolation (DLPP) and the directed polymer (DP).
The Greene's theorem \cite{greene74} (see for example \cite{sagan00} for a modern exposition) characterises the shape of the output tableaux as lengths of longest non-decreasing subsequences.
As an immmediate consequence, this can be viewed as the RSK algorithm transforming a matrix to a multilayer non-intersecting generalisation of the DLPP. Specifically the first row of the output tableaux corresponds to precisely the DLPP.
When randomness is introduced into the input matrix, this connection yields exact formulas for the distribution of DLPP in geometric and exponential environments \cite{johansson00}.
The geometric lifting of the DLPP is the partition function of the directed polymer (DP) where the solvable environment is that of the log-Gamma weights \cite{seppalainen12}.
And unsurprisingly the gRSK is related to the DP the same way as RSK is related to the DLPP.
This was used in \cite{corwin-oconnell-seppalainen-zygouras14} to obtain exact formulas for the distribution of the partition function of DP in a log-Gamma environment.

The DLPP and DP can be defined locally using a similar growth rule to the (g)RSK. And given the solvable environment there present reversibility results of this local growth rule of the partition function called the Burke property.
It is used to show the cube root variance fluctuations of the partition functions \cite{balazs-cator-seppalainen06,seppalainen12}.

Also in these solvable models, the distribution of the shape of the tableaux are related to certain special functions.
In the RSK setting it is the Schur measure \cite{oconnell03a}, related to the Schur functions, and in gRSK setting it is the Whittaker measure, related to the $\mathfrak{gl}_{\ell + 1}$-Whittaker functions \cite{corwin-oconnell-seppalainen-zygouras14}.
This kind of results can often be obtained using a combination of Doob's $h$-transform and the Markov function theorem \cite{rogers-pitman81}.

In \cite{oconnell-seppalainen-zygouras14} a reformulation of the gRSK called the local moves was used to give a more direct treatment than the Markov function theorem to show the connection between the gRSK and the Whittaker functions.

The local moves can be generalised to take an array of Young diagram shape.
This idea can be found in the proof of the Two Polytope Theorem in  \cite{pak01}.
\footnote{See also the historical remarks in Section 8 of \cite{pak01}. An exposition of this idea can be also found in \cite{hopkins14}.}

In \cite{nguyen-zygouras16} this idea was used to yield the joint laws of the partition functions of the log-Gamma polymer in the space-like direction. 
Specifically it was used to formulate a geometric version of the multilayer polynuclear growth model (PNG) introduced in \cite{johansson03}, from which the joint law of the polymer partition functions at a fixed time could be written down.

One direction for generalisation of the (g)RSK algorithms is to interpolate using $q$-deformation.
The Macdonald polynomials were introduced in \cite{macdonald88}. See \cite{macdonald98} for a detailed introduction.
They are symmetric polynomials of two parameters $q$ and $t$.
We only consider $t = 0$, in which case they are also the $q$-Whittaker functions with some prefactors, as they are eigenfunctions of the $q$-deformed quantum Toda Hamiltonian \cite{gerasimov-lebedev-oblezin10}.
On the one hand the $q$-Whittaker functions interpolate between the Schur functions ($q = 0$) and the Whittaker functions ($q \to 1$ with proper scalings \cite{gerasimov-lebedev-oblezin12}).
On the other hand the simiarlity among structures of the Macdonald polynomials, Schur polynomials and the Whittaker functions makes the Macdonald processes and measures \cite{forrester-rains02,borodin-corwin14} possible.
This motivates the search for $q$-deformed RS(K) algorithms.

The $q$RS algorithms were introduced in \cite{oconnell-pei13} (column insertion version) and in \cite[Dynamics 3, $h = (1, 1, \dots, 1)$]{borodin-petrov13} (row insertion version). 
In \cite{matveev-petrov15} several $q$-deformed RSK ($q$RSK) algorithms were introduced. 
In all these $q$-deformations the algorithms transform inputs into random pairs of tableaux, rather than just one pair of tableaux.
These $q$-algorithms all have the desired property of transforming the input into various $q$-Whittaker processes.

In this article we work on the $q$RSK row insertion algorithm introduced in \cite[Section 6.1 and 6.2]{matveev-petrov15}.
It was shown in that paper that the $q$RSK algorithm transforms a matrix with $q$-geometric weights into the $q$-Whittaker process, and the push-forward measure of the shape of the output tableaux is the $q$-Whittaker measure.

We give the Noumi-Yamada description of this algorithm, from which we obtain a branching growth diagram construction similar to that in \cite{pei14}, and show that the algorithm is symmetric:
\begin{thm}\label{t:qsym}
  Let $\phi_A(P, Q) = \prob(q\RSK(A) = (P, Q))$ be the probability of obtaining the tableau pair $(P, Q)$ after performing $q$RSK on matrix $A$, then
  \begin{align*}
    \phi_A(P, Q) = \phi_{A'}(Q, P)
  \end{align*}
  where $A'$ is the transpose of $A$.
\end{thm}
We also formulate a $q$-polymer model which corresponds to the first row of the output tableaux of the $q$RSK. 
It interpolates between the DLPP ($q = 0$) and the DP ($q \to 1$ with proper scaling).
The Burke's property also carries over to the $q$-setting naturally, with which one immediately obtains some asymptotic results for the $q$-polymer with stationary boundary conditions. 
See Section \ref{s:qpolymer} for more details.
Also see Section \ref{s:qdef} for definition of $(x; q)_\infty$ that appears in the theorem.

\begin{thm}\label{t:lln}
  Let $Z$ be the partition function of the $q$-polymer.
  With stationary boundary conditions defined in Section \ref{s:qburke},
  \begin{align}
    \expe Z(\ell, j) = \ell \gamma(\alpha) + j \gamma(\beta) \label{eq:qpolyexp}
  \end{align}
  where 
  \begin{align*}
    \gamma(x) = x (\log E_q)'(x)
  \end{align*}
  where $E_q (x) = (x; q)_\infty^{-1}$.
  Moreover, almost surely
  \begin{align}\label{eq:lln}
    \lim_{N \to \infty} {Z(\lfloor N x \rfloor, \lfloor N y \rfloor) \over N} = x \gamma(\alpha) + y \gamma(\beta).
  \end{align}
\end{thm}

Finally we formulate a $q$-local move that agrees with the $q$RSK when taking a matrix input.
Like in \cite{hopkins14,nguyen-zygouras16}, we use the $q$-local moves to generalise the $q$RSK to take arrays on a Young diagram as the input, propose the corresponding $q$PNG model, and write down the joint distribution of the $q$-polymer partition functions in the space-like direction.

Like in \cite{oconnell-seppalainen-zygouras14,nguyen-zygouras16}, the basic operation of the $q$-local moves, called $\rho_{n,k}$, works on diagonal strips $(i, j)_{i - j = n - k}$ of the input.

In those two papers, when the gRSK is defined as a composition of the $\rho_{n, k}$, they are defined in two different ways, row-by-row or antidiagonal-by-antidiagonal.

In \cite{hopkins14}, $\rho_{n, k}$ (or more precisely the map $b_{n, k}$ in \cite[(3.5)]{oconnell-seppalainen-zygouras14}, see also \eqref{eq:t1}\eqref{eq:t2}) were referred to as ``toggles''.
It was shown there the map called $\mathcal{RSK}$ can be of any composition of the toggles whenever they agree with a growth sequence of the underlying Young diagram of the input array.

In this article, we generalise this to the $q$-setting. 
By identifying the input pattern as an array on a Young diagram $\Lambda$, we show that the $q$RSK map $T_\Lambda$ can be of any composition of the $\rho$'s whenever they agree with a growth sequence of $\Lambda$.
For details of definitions of $\rho_{n, k}$ and $T_\Lambda$ see Section \ref{s:qlocalmoves}.

We fit the input $(w_{i, j})_{(i, j) \in \Lambda}$ into an infinite array $A = (w_{i, j})_{i, j \ge 1} \in \pint^{\pint_+ \times \pint_+}$ and define $T_\Lambda$ such that when acting on an infinite array like $A$ it only alters the topleft $\Lambda$ part of the array.

Let\footnote{See Section \ref{s:notations} for explanation of notations like $a : b$}
\begin{align*}
  r_1 &= t_{\Lambda'_1, 1} \\
  r_j &= \sum_{k = 1 : (j \wedge \Lambda'_j)} t_{\Lambda'_j - k + 1, j - k + 1} - \sum_{k = 1 : ((j - 1) \wedge \Lambda'_j)} t_{\Lambda'_j - k + 1, j - k}, \qquad j = 2 : \Lambda_1 \\
  \hat r_1 &= t_{1, \Lambda_1} \\
  \hat r_i &= \sum_{k = 1 : (i \wedge \Lambda_i)} t_{i - k + 1, \Lambda_i - k + 1} - \sum_{k = 1 : ((i - 1) \wedge \Lambda_i)} t_{i - k, \Lambda_i - k + 1} \qquad i = 2 : \Lambda'_1
\end{align*}
Given a $q$-geometrically distributed input array, we can write down the explicit formula of the push-forward measure of $T_\Lambda$.

In this article we let $(\hat\alpha_i)$ and $(\alpha_j)$ be parameters such that $\hat \alpha_i \alpha_j \in (0, 1)$ for all $i, j$.
Also note for integer $n$, denote $(n)_q$ to be the $q$-Pochhammer $(q; q)_n$ (see Section \ref{s:qdef}).
\begin{thm}\label{t:lmpush}
  Given that the input pattern $(w_{i, j})_{(i, j)}$ have independent $q$-geometric weights
\begin{align*}
  w_{i, j} \sim \qGeom(\hat\alpha_i \alpha_j), \qquad \forall i, j
\end{align*}
the distribution of $T_\Lambda A(\Lambda)$ is
\begin{align*}
  \prob(T_\Lambda A(\Lambda) &= (t_{i, j})_{(i, j) \in \Lambda}) \\
  &= \mu_{q, \Lambda}(t) := (t_{1 1})_q^{-1} {\prod_{(i, j) \in \Lambda: (i - 1, j - 1) \in \Lambda} (t_{i j} - t_{i - 1, j - 1})_q \over \prod_{(i, j) \in \Lambda: (i, j - 1) \in \Lambda} (t_{i j} - t_{i, j - 1})_q \prod_{(i, j) \in \Lambda: (i - 1, j) \in \Lambda} (t_{i j} - t_{i - 1, j})_q} \\
  &\;\;\;\;\;\;\;\;\;\;\;\;\times \alpha^r \hat\alpha^{\hat r} \prod_{(i, j) \in  \Lambda}  (\hat \alpha_i \alpha_j; q)_\infty \ind_{t \in D_\Lambda},
\end{align*}
where
\begin{align*}
  D_\Lambda := \{t \in \pint^\Lambda: t_{i - 1, j} \le t_{i, j} \forall \{(i, j), (i - 1, j)\} \subset \Lambda, t_{i, j - 1} \le t_{i, j} \forall \{(i, j), (i, j - 1)\} \subset \Lambda\}.
\end{align*}
\end{thm}

We define an outer corner of a Young diagram to be any cell without neighbours to the right of below itself.
More precisely, $(n, m)$ is an outer corner of $\lambda$ if $\lambda_n = m$ and $\lambda_{n + 1} < m$.

Given a Young diagram $\Lambda$ with outer corners $(n_1, m_1), (n_2, m_2), \dots, (n_p, m_p)$, summing over the non-outer-corner points we can show the multipoint distribution of the $q$-polymer:
\begin{cly}\label{c:qpdist}
  For $m_1 \le m_2 \le \dots \le m_p$ and $n_1 \le n_2 \le \dots \le n_p$.
  Let $\Lambda$ be the Young diagram with outer corners $((n_i, m_{p - i + 1}))_{i = 1 : p}$.
  The partition functions $(Z(n_1, m_p), Z(n_2, m_{p - 1}), \dots, Z(n_p, m_1))$ of the $q$-polymer in a $(\hat\alpha, \alpha)$-$q$-geometric environment has the following joint distribution:
  \begin{align*}
    \prob(Z(n_1, m_p) = x_1, Z(n_2, m_{p - 1}) = x_2, \dots, Z(n_p, m_1) = x_p) = \sum_{t \in D_\Lambda, t_{n_i, m_{p - i + 1}} = x_i \forall i = 1 : p} \mu_{q, \Lambda} (t)
  \end{align*}
\end{cly}

Furthermore, if we specify $m_i = n_i = i$ for $i = 1 : p$, that is, $\Lambda$ is a staircase Young diagram, then we may define a $q$PNG multilayer noncolliding process, as in \cite{johansson03,nguyen-zygouras16}.
By recognising $p$ as the time, we can write down the joint distribution of the partition functions of the $q$-polymer at time $p$.

\begin{cly}\label{c:qpng}
  The partition functions of the $q$-polymer at time $p$ with the $(\hat\alpha, \alpha)$-$q$-geometric environment has the following joint distribution
  \begin{align*}
    \prob(Z(1, p) &= x_1, Z(2, p - 1) = x_2, \dots, Z(p, 1) = x_p) \\
    &= \prod_{i + j \le p + 1} (\hat \alpha_i \alpha_j; q)_\infty \sum_{t \in D_\Lambda, t_{i, p - i + 1} = x_i, \forall i = 1 : p} \left((t_{11})_q^{-1} {\prod_{i + j \le p - 1} (t_{i + 1, j + 1} - t_{i, j})_q \over \prod_{i + j \le p} \left( (t_{i + 1, j} - t_{i, j})_q (t_{i, j + 1} - t_{i, j})_q \right)}\right. \\
    &\left.\qquad\qquad\qquad\qquad\qquad\qquad\qquad\qquad\qquad\qquad\qquad\qquad\times{\prod_{i + j = p + 1} (\hat \alpha_i \alpha_j)^{t_{i, j}} \over \prod_{i, j > 1, i + j = p + 2} (\hat \alpha_i \alpha_j)^{t_{i - 1, j - 1}}}\right)
  \end{align*}
\end{cly}

If we restrict to $p = 1$, that is, $\Lambda$ is a rectangular Young diagram, we obtain the following (recall the definition of $r$ and $\hat r$ just before Theorem \ref{t:lmpush}):

\begin{cly}
  Given a $q$-geometric distributed matrix $(w_{i, j} \sim \qGeom(\hat\alpha_i \alpha_j))_{i = 1 : n, j = 1 : m}$, the push-forward measure of $q$RSK taking this matrix is
  \begin{align*}
    \mu_q (t) &= (t_{11})_q^{-1} {\prod_{i = 2 : n, j = 2 : m} (t_{i, j} - t_{i - 1, j - 1})_q \over \prod_{i = 1 : n, j = 2 : m}(t_{i, j} - t_{i, j - 1})_q \prod_{i = 2 : n, j = 1 : m} (t_{i, j} - t_{i - 1, j})_q} \\
    &\times \alpha^r \hat\alpha^{\hat r} \prod_{i = 1 : n, j = 1 : m} (\hat\alpha_i \alpha_j; q)_\infty \ind_{t \in D_\Lambda}
  \end{align*}
\end{cly}

By summing over all $t_{i, j}$ with fixed diagonals $(t_{n, m}, t_{n - 1, m - 1}, \dots, t_{(n - m)^+ + 1, (m - n)^+ + 1}) = (\lambda_1, \dots, \lambda_{m \wedge n})$, we recover a result in \cite{matveev-petrov15}:

\begin{cly}\label{c:qwmeasure}
  Given a $q$-geometric distributed matrix $(w_{i, j} \sim \qGeom(\hat\alpha_i \alpha_j))_{i = 1 : n, j = 1 : m}$, the shape of the output tableaux after applying $q$RSK on $(w_{i, j})$ is distributed according to the $q$-Whittaker measure:
  \begin{align*}
    \prob((t_{n, m}, t_{n - 1, m - 1}, \dots, t_{(n - m)^+ + 1, (m - n)^+ + 1}) = (&\lambda_1, \dots, \lambda_{m \wedge n})) \\
    &= P_\lambda(\alpha) Q_\lambda(\hat\alpha) \prod_{i = 1 : n, j = 1 : m} (\hat\alpha_i \alpha_j; q)_\infty,
  \end{align*}
where $P_\lambda$ and $Q_\lambda$ are $(t = 0)$-Macdonald polynomials.
\end{cly}

The rest of the article is organised as follows. 
In Section \ref{s:sym} we review some preliminaries on (g)RSK, $q$-deformations and the $q$-Whittaker functions. Then we give the Noumi-Yamada description and growth diagram formulations of the $q$RSK algorithm, with which we prove the symmetry property Theorem \ref{t:qsym}. 
In Section \ref{s:qpolymer} we formulate the $q$-polymer, define and discuss the Burke relations, prove the $q$-Burke property, with which we prove Theorem \ref{t:lln}
In Section \ref{s:qlocalmoves} we formulate the $q$-local moves, show their relation to the $q$RSK, prove Theorem \ref{t:lmpush}, propose the $q$PNG, and discuss a measure on the matrix and its classical limit to a similar measure in \cite{oconnell-seppalainen-zygouras14}.

\subsection{Notations}\label{s:notations}
We list some notations we use in this article.
\begin{itemize}
  \item $\pint$ is the set of the nonnegative integers, and $\pint_+$ the set of the positive integers.
    \item $\ind_A$ is the indicator function on the set $A$.
    \item $a : b$ is $\{a, a + 1, \dots, b\}$
    \item $[n]$ is $1 : n$
      \item $i = a : b$ means $i \in \{a, a + 1, \dots, b\}$
      \item $(\lambda_{1 : m})$ denotes $(\lambda_1, \lambda_2, \dots, \lambda_m)$.
      \item $w_{n, 1 : k}$ is $(w_{n, 1}, w_{n, 2}, \dots, w_{n, k})$
      \item $w_{1 : n, 1 : m}$ is a matrix $(w_{i, j})_{n \times m}$
        \item $:=$ means (re)evaluation or definition. For example $u := u + a$ means we obtain a new $u$ which has the value of the old $u$ adding $a$.
        \item For $a = (a_{1 : m})$, $b = (b_{1 : m})$, $a^b := \prod_{i = 1 : m} a_i^{b_i}$.
        \item For an array $A = (w_{i, j})_{(i, j) \in \pint_+^2}$ and a index set $I$, $A(I) := (w_{i, j})_{(i, j) \in I}$.
          \item $\ve_i$ is the vector with a $1$ at the $i$th entry and $0$ at all other entries: $\ve_i = (0, 0, \dots, 0, 1, 0, 0,\dots)$.
\end{itemize}

\noindent\textbf{Acknowledgement.} The author would like to thank Ziliang Che for fruitful discussions. 
The author would also like to acknowledge communication and discussions with Vu-Lan Nguyen, Jeremy Quastel, Neil O'Connell, Dan Romik, Konstantin Matveev, Timo Sepp\"al\"ainen, Alexei Borodin, Nikos Zygouras and Ivan Corwin.
Furthermore the author would like to thank an anonymous reviewer for reading the article and providing valuable feedbacks which results in much improvement of the article.
This work was supported by the Center of Mathematical Sciences and Applications at Harvard University.

\section{Noumi-Yamada description, growth diagrams and the symmetry property}\label{s:sym}
In this section we review the basics of the theory of Young tableaux and the Noumi-Yamada description of the (g)RSK. We also review some $q$-deformations and related probability distributions. Then we formulate the Noumi-Yamada description and growth diagram for the $q$RSK, and show how to use the latter to prove the symmetry property Theorem \ref{t:qsym}.
\subsection{A review of the RSK and gRSK algorithms}

A Young diagram $\lambda = (\lambda_1, \lambda_2, \dots, \lambda_m)$ is a nonincreasing nonnegative integer sequence.

One may represent a Young diagram as a collection of coordinates in $\pint_+^2$.
More specifically, in this representation $\lambda = \{(i, j): \lambda_i \ge j\}$.
For example $\lambda = (4, 3, 1, 1)$ has the 2d coordinate representation $\{(1, 1), (1, 2), (1, 3), (1, 4), (2, 1), (2, 2), (2, 3), (3, 1), (4, 1)\}$, and it can be visualised as follows, where we labelled some coordinates:
\begin{center}
  \includegraphics{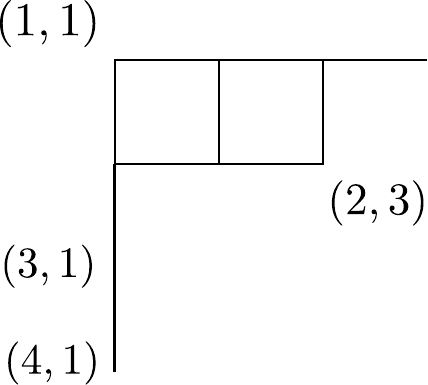}
\end{center}
We use these two representations without specifying under unambiguous contexts.
For example, an array restricted to $\lambda$ and denoted as $A(\lambda)$ uses the 2d coordinates representation.

A Gelfand-Tsetlin (GT) pattern is a triangular array $(\lambda^k_j)_{1 \le j \le k \le m}$ satisfying the interlacing constraints, that is
\begin{align*}
  \lambda^k \prec \lambda^{k + 1},
\end{align*}
where $a \prec b$ means
\begin{align*}
  b_1 \ge a_1 \ge b_2 \ge a_2 \ge \dots.
\end{align*}
The exact constraints of the GT pattern are thus
\begin{align*}
  \lambda^{k + 1}_{j + 1} \le \lambda^k_j \le \lambda^{k + 1}_j \qquad \forall k \ge j \ge 1
\end{align*}

We refer to the indices of the GT pattern coordinates in the following way.
Given a coordinate $\lambda^k_j$, we call the superscript ($k$ here) the level, the subscript ($j$ here) the edge.
Later when we consider the time evolution of the GT patterns, we use an argument in the bracket to denote time.
Therefore $\lambda^k_j (\ell)$ is the coordinate at time $\ell$, $k$th level and $j$th edge.

We visualise a GT pattern, for example with 5 levels as follows, where we also annotate the interlacing relations:
\begin{center}
\includegraphics[scale=.3]{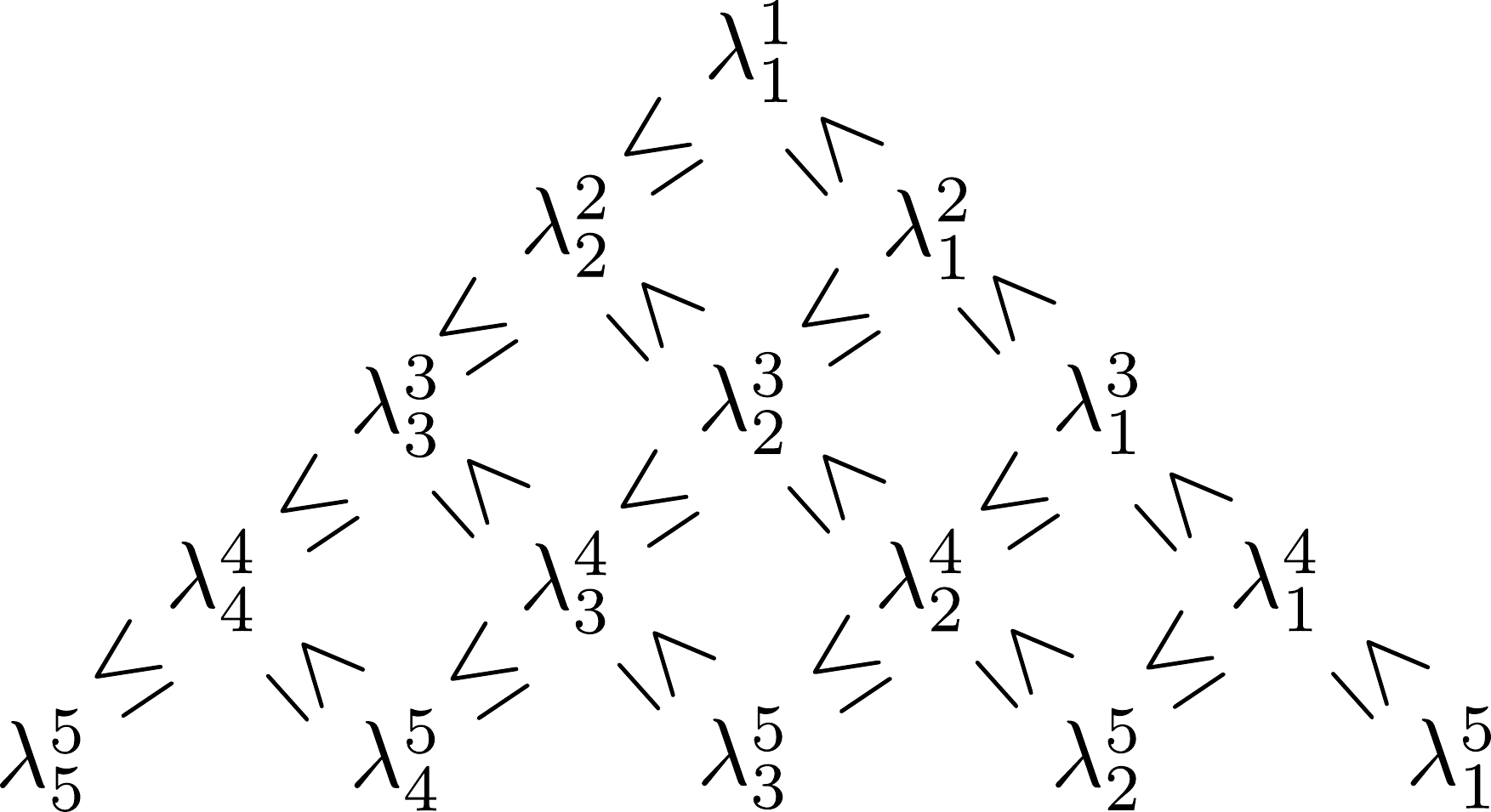}
\end{center}
So the levels correspond to rows and edges corresponds to edges from the right in the picture.

Throughout this article we do not take powers of $\lambda_j$ so the superscript on $\lambda$ is always an index rather than a power.
The same applies to notations $a^j_k$ in Noumi-Yamada description of the $q$RSK, as well as the $t$ in the proof of Theorem \ref{t:qlocalmoves}.

A semi-standard Young tableau, which we simply refer to as a tableau, is an Young diagram-shaped array filled with positive integers that are non-descreasing along the rows and increasing along the columns.
The underlying Young diagram is called the shape of the tableau.

A tableau $T$ corresponds to a GT pattern $(\lambda^k_j)$ in the following way:
\begin{align*}
  \lambda^k_j = \#\left\{\text{Number of entries in row }j\text{ that are no greater than }k \right\}
\end{align*}

For example 
$\begin{array}{ccccc}
  1&2&2&3&3\\2&3&4&\\4&&&
\end{array}$
is a tableau with shape $\lambda = (5, 3, 1)$ and GT pattern 
\begin{align*}
(\lambda^1_1, \lambda^2_1, \lambda^2_2, \lambda^3_1, \lambda^3_2, \lambda^3_3, \lambda^4_1, \lambda^4_2, \lambda^4_3, \lambda^4_4) = (1, 3, 1, 5, 2, 0, 5, 3, 1, 0)
\end{align*}

In this article we work on the GT patterns only since it is easier to manipulate symbolically than tableaux.
We use the terms GT patterns and tableaux interchangeably hereafter.

Clearly the shape of a tableau $(\lambda^k_j)_{1 \le j \le k \le m}$ is the bottom row in the visualisation of the GT pattern $\lambda^m = (\lambda^m_1, \lambda^m_2, \dots, \lambda^m_m)$.

The RSK algorithm takes in a matrix $A = (w_{i j})_{n\times m}$ as the input and gives a pair of tableaux $(P, Q)$ as the output.
We call the output $P$-tableau the insertion tableau, and the $Q$-tableau the recording tableau.
When we mention the output tableau without specifying, it is by default the $P$-tableau, as most of the time we focus on this tableau.
We usually denote the corresponding GT pattern for the $P$- and $Q$- tableaux as respectively $(\lambda^k_j)$ and $(\mu^k_j)$.

The RSK algorithm is defined by the insertion of a row $(a^1, a^2, \dots, a^m) \in \pint^m$ of nonnegative integers into a tableau $(\lambda^j_k)$ to obtain a new tableau $(\tilde\lambda^j_k)$.
We call this insertion operation the RSK insertion and postpone its exact definition to Definition \ref{d:nyrsk}.

When applying the RSK algorithm to a matrix $w_{1 : n, 1 : m}$, we start with an empty tableau $\lambda^k_j (0) \equiv 0$, and insert $w_{1, 1 : m}$ into $(\lambda^k_j (0))$ to obtain $(\lambda^k_j (1))$, then insert $w_{2, 1 : m}$ into $(\lambda^k_j (1))$ to obtain $(\lambda^k_j (2))$ and so on and so forth.
The output $P$-tableau is the GT pattern at time $n$: $\lambda^k_j = \lambda^k_j (n)$, and the $Q$-tableau is the GT pattern at level $m$: $\mu^\ell_j = \lambda^m_j (\ell)$.

For a traditional definition of the RSK insertion, see e.g. \cite{fulton97}.
The definition we give here is the Noumi-Yamada description.

\begin{dfn}[The Noumi-Yamada description of the RSK insertion]\label{d:nyrsk}
Suppose at time $\ell - 1$ we have a tableau $(\lambda^j_k) = (\lambda^j_k (\ell - 1))$ and want to RSK-insert row $w_{\ell, 1 : m}$ into it to obtain a new tableau $(\tilde\lambda^j_k) = (\lambda^j_k (\ell))$. 

This is achieved by initialising $a^{1 : m}_1 = w_{\ell, 1 : m}$ and recursive application (first along the edges, starting from $1$ and incrementing, then along the levels, starting from the edge index and incrementing) of the following
\begin{align*}
\tilde\lambda^k_k &= \lambda^k_k + a^k_k\\
\tilde\lambda^j_k &= a^j_k + (\lambda^j_k \vee \tilde\lambda^{j - 1}_k), \qquad j > k \\
a^j_{k + 1} &= a^j_k + \lambda^j_k - \tilde\lambda^j_k + \tilde\lambda^{j - 1}_k - \lambda^{j - 1}_k
\end{align*}
\end{dfn}

The Noumi-Yamada description does not rely on $w_{ij}$ being integers and hence extends the RSK algorithm to take real inputs.
Similarly one can define the Noumi-Yamada description for the gRSK algorithm, which is simply a geometric lifting of the RSK algorithm.
It also takes real inputs.
\begin{dfn}[The Noumi-Yamada description for the gRSK algorithm]
Suppose at time $\ell - 1$ we have a tableau $(z^j_k) = (z^j_k (\ell - 1))$ and want to gRSK-insert a row $w_{\ell, 1 : m}$ into it to obtain a new tableau $(\tilde z^j_k) = (z^j_k (\ell))$. 

This is done by initialising $(a^i_1)_{i = 1 : m} = (e^{w_{\ell, i}})_{i = 1 : m}$ and the recursive application (in the same way as in the RSK insertion) of the following
\begin{align*}
\tilde z^k_k &= z^k_k a^k_k\\
\tilde z^j_k &= a^j_k (z^j_k + \tilde z^{j - 1}_k), \qquad j > k \\
a^j_{k + 1} &= a^j_k {z^j_k \tilde z^{j - 1}_k \over \tilde z^j_k z^{j - 1}_k}
\end{align*}
\end{dfn}

Before defining the $q$RSK algorithm, let us review some $q$-deformations.
\subsection{$q$-deformations}\label{s:qdef}
A good reference of the $q$-deformations is \cite{gasper-rahman04}. In this article we assume $0 \le q < 1$.

Define the $q$-Pochhammers and the $q$-binomial coefficients as
\begin{align*}
  (\alpha; q)_k &= 
  \begin{cases}
    \prod_{i = 0 : k - 1} (1 - \alpha q^i) & k > 0\\
    1 & k = 0 \\
    \prod_{i = 1 : - k} (1 - \alpha q^{-i})^{-1} & k < 0
  \end{cases}\\
  (k)_q &= (q; q)_k \\
  {n \choose k}_q &= {(n)_q \over (k)_q (n - k)_q}
\end{align*}

We also define three $q$-deformed probability distributions.

\subsubsection{$q$-geometric distribution}
\begin{dfn}
  Given $\alpha \in (0, 1)$, a random variable $X$ is said to be distributed according to the $q$-geometric distribution $\qGeom(\alpha)$ if it has probability mass function (pmf)
  \begin{align*}
    f_X(k) = {\alpha^k \over (k)_q} (\alpha; q)_\infty, \qquad k = 0, 1, 2, \dots
  \end{align*}
\end{dfn}
The first moment of the $q$-geometric distribution with parameter $\alpha$ is
\begin{align}
  \sum_k {k \alpha^k \over (k)_q} (\alpha; q)_\infty = (\alpha; q)_\infty \alpha \partial_{\alpha} \sum {\alpha^k \over (k)_q} = \alpha (\log E_q)'(\alpha). \label{eq:qgeommoment}
\end{align}
where $E_q(\alpha) = (\alpha; q)_\infty^{-1}$ is a $q$-deformation of the exponential function (see for example (1.3.15) of \cite{gasper-rahman04}).

\subsubsection{$q$-binomial distribution}
There are several $q$-deformations of the binomial distribution. 
The one that is used in \cite{matveev-petrov15} to construct the $q$RSK is also called $q$-Hahn distribution. 
It appeared in \cite{povolotsky13}.
Apart from the dependency on $q$, it is has three parameters ($\xi, \eta, n$).
For $0 \le \eta \le \xi < 1$, and $n \in \pint \cup \{\infty\}$, the pmf is
  \begin{align*}
    \phi_{q, \xi, \eta} (k | n) = \xi^k {(\eta / \xi; q)_k (\xi; q)_{n - k} \over (\eta; q)_n} {n \choose k}_q, \qquad k = 0 : n
  \end{align*}
  The fact that it is a probability distribution can be found in, for example \cite[Exercise 1.3]{gasper-rahman04}.

\subsubsection{$q$-hypergeometric distribution}
The $q$-hypergeometric distribution we consider here is defined as follows. 
For $m_1, m_2, k \in \pint$ with $m_1 + m_2 \ge k$, $X \sim \qHyp(m_1, m_2, k)$ if the pmf of $X$ is
\begin{align*}
f_X(\ell) = q^{(m_1 - \ell)(k - \ell)}{{m_1 \choose \ell}_q {m_2 \choose k - \ell}_q \over {m_1 + m_2 \choose k}_q}
\end{align*}
The corresponding $q$-Vandermonde identity 
\begin{align*}
  \sum_\ell q^{(m_1 - \ell) (k - \ell)} {m_1 \choose \ell}_q {m_2 \choose k - \ell}_q = {m_1 + m_2 \choose k}_q
\end{align*}
can be proved directly by writing $(1 + x) (1 + q x) \cdots (1 + q^{m_1 + m_2 - 1} x)$ in two different ways.

As with the usual hypergeometric distribution, the support of $\qHyp(m_1, m_2, k)$ is 
\begin{align}
  0 \vee (k - m_2) \le \ell \le m_1 \wedge k \label{eq:qhypsupp}
\end{align}

When $m_2 = \infty$, the distribution is symmetric in $m_1$ and $k$:
\begin{align*}
  f_{\qHyp(m_1, \infty, k)} (\ell) = f_{\qHyp(k, \infty, m_1)} (\ell) = q^{(m_1 - \ell)(k - \ell)} {(m_1)_q (k)_q \over (\ell)_q (m_1 - \ell)_q (k - \ell)_q}, \qquad 0 \le \ell \le m_1 \wedge k
\end{align*}
This distribution appeared in \cite{blomqvist52}.

When $k = 0$ or $m_1 = 0$, by \eqref{eq:qhypsupp} the distribution is supported on $\{0\}$:
\begin{align}
  f_{\qHyp(0, m_2, k)} (s) = f_{\qHyp(m_1, m_2, 0)} (s) = \ind_{s = 0}. \label{eq:qhypk=0}
\end{align}

The fact that the $q$Hyp is a probability distribution yields the following identities, where the second follows by taking $m_2 = \infty$:

\begin{align}
  \sum_{s} q^{(m_1 - s)(k - s)} (m_1 - s)_q^{-1} (k - s)_q^{-1} (s)_q^{-1} (m_2 - k + s)_q^{-1} &= {(m_1 + m_2)_q \over (m_1)_q (m_2)_q (k)_q (m_1 + m_2 - k)_q}
  \label{eq:qhyp}\\
  \sum_{s} q^{(m_1 - s)(k - s)} (m_1 - s)_q^{-1} (k - s)_q^{-1} (s)_q^{-1} &= (m_1)_q^{-1} (k)_q^{-1} \label{eq:qhypinf}
\end{align}

\subsubsection{From the $q$-binomial distribution to the $q$-hypergeometric distribution}
The $q$-binomial distribution is related to the $q$-hypergeometric distribution in the following way:
\begin{lem}\label{l:qbinqhyp}
  For nonnegative integers $a \le b \ge c$, let $X$ be a random variable distributed according to $\phi_{q^{-1}, q^a, q^b} (\cdot | c)$, then $c - X$ is distributed according to $\qHyp(c, b - c, a)$.
\end{lem}
\begin{proof}
  By the definition of the $q$-hypergeometric distribution it suffices to show
\begin{align}
  \phi_{q^{-1}, q^a, q^b}(s | c) = q^{s (s + a - c)} {b \choose a}_q^{-1} {c \choose s}_q {b - c \choose s + a - c}_q \label{eq:qbinhyp}
\end{align}
First we apply $(x; q^{-1})_n = (- 1)^n x^n q^{-{n \choose 2}} (x^{-1}; q)_n$ and ${n \choose k}_{q^{-1}} = q^{- k (n - k)} {n \choose k}_q$ to the left hand side to turn the $q^{-1}$-Pochhammers into $q$-Pochhammers.
The left hand side thus becomes
\begin{align*}
q^{(a - b)(c - s)} {(q^{a - b}; q)_s (q^{-a}; q)_{c - s} \over (q^{-b}; q)_c} {c \choose s}_q.
\end{align*}
Furthermore using $(q^{-n}; q)_k = {(n)_q \over (n - k)_q} (-1)^k q^{{k \choose 2} - n k}$ the above formula becomes the right hand side of \eqref{eq:qbinhyp}.
\end{proof}

\subsubsection{The ($t = 0$)-Macdonald polynomials and the $q$-Whittaker functions}
Let us define the $(t = 0)$-Macdonald polynomials. 
For $x = (x_{1 : N})$ and $\lambda = (\lambda_{1 : \ell})$ with $\ell \le N$, we redefine $\lambda$ by padding $N - \ell$ zeros to the end of it:
\begin{align*}
  \lambda := (\lambda_1, \lambda_2, \dots, \lambda_\ell, 0, 0, \dots, 0).
\end{align*}
Given a tableau $(\lambda^k_j)$ define its type $\ty((\lambda^k_j))$ by
\begin{align*}
  \ty((\lambda^k_j))_i :=
  \begin{cases}
    \lambda^1_1 & i = 1 \\
    \sum_{\ell = 1 : i} \lambda^i_\ell - \sum_{\ell = 1 : i - 1} \lambda^{i - 1}_\ell & i > 1
  \end{cases}
\end{align*}
Then the $(t = 0)$-Macdonald polynomials of rank $N - 1$ indexed by $\lambda$ and the $q$-Whittaker function $\psi_x(\lambda)$ are defined as
\begin{align*}
  P_\lambda (x) &= \sum_{(\lambda^k_j)_{1 \le j \le k \le N}, \lambda^{k - 1} \prec \lambda^k \forall k, \lambda^N = \lambda} x^{\ty((\lambda^k_j))} \prod_{1 \le j < k \le N} {\lambda^k_j - \lambda^k_{j + 1} \choose \lambda^k_j - \lambda^{k - 1}_j}_q,\\
  Q_\lambda (x) &= (\lambda_N)_q^{-1} P_\lambda(x) \prod_{i = 2 : N} (\lambda_i - \lambda_{i - 1})_q^{-1},\\
  \psi_x (\lambda) &= (\lambda_N)_q Q_\lambda (x).
\end{align*}
The $q$-Whittaker measure discussed in this article is the one induced by the Cauchy-Littlewood identity:
\begin{align*}
  \mu_{q\text{-Whittaker}} (\lambda) = P_\lambda (\alpha) Q_\lambda (\hat\alpha) \prod_{i j} (\hat\alpha_i \alpha_j; q)_\infty.
\end{align*}

\subsubsection{Classical limits}
In this section we let $q = e^{- \epsilon}$ for small $\epsilon > 0$ and collect some results concerning the classical limits.

Let
\begin{align*}
  A(\epsilon) = - {\pi^2 \over 6} \epsilon^{-1} - {1 \over 2} \log \epsilon + {1 \over 2} \log 2 \pi \\
\end{align*}

\begin{lem}\label{l:qclassicallimits} Let $q = e^{- \epsilon}$ and $m(\epsilon) = \epsilon^{-1} \log \epsilon^{-1}$, then
  \begin{enumerate}
    \item $(q^t; q)_\infty = \Gamma(t)^{- 1} \exp(A(\epsilon) + (1 - t) \log \epsilon + O(\epsilon))$. Specifically, $(q; q)_\infty = \exp(A(\epsilon) + O(\epsilon))$
    \item For $\alpha \ge 1$,
    $f_\alpha(y) := (\lfloor\alpha m(\epsilon) + \epsilon^{-1} y\rfloor)_q = 
      \begin{cases}
      \exp(A(\epsilon) + e^{- y} + O(\epsilon)) & \alpha = 1 \\
      \exp(A(\epsilon) + O(\epsilon)) & \alpha > 1
      \end{cases}$
    \item $\log (\lfloor \epsilon^{-1} y \rfloor)_q = \epsilon^{-1} \left( \Li_2(e^{- y}) - {\pi^2 \over 6} \right) + o(\epsilon^{-1})$, where 
      \begin{align*}
        \Li_2(x) = \sum_{n \ge 1} {x^n \over n^2} = - \int_0^x {\log(1 - u) \over u} du
      \end{align*}
      is the dilogarithm function.
  \end{enumerate}
\end{lem}

Item 1 can be found, for example as a special case of Theorem 3.2 in \cite{banerjee-wilkerson16}. Item 2 was proved as Lemma 3.1 of \cite{gerasimov-lebedev-oblezin12}. 
Item 3 can be derived as follows:
\begin{align*}
  \epsilon \log(\lfloor \epsilon^{-1} y \rfloor)_q = \epsilon \sum_{k = 1 : \lfloor \epsilon^{-1} y \rfloor} (1 - e^{-\epsilon k}) \approx \int_0^y \log (1 - e^{-t}) dt = \Li_2(e^{- y}) - {\pi^2 \over 6}.
\end{align*}

\subsection{The Noumi-Yamada description of the $q$RSK}
Now we can define a Noumi-Yamada description for the $q$RSK.

Throughout this article we adopt the convention that for any Young diagram $\lambda$, the $0$th edge are $\infty$: $\lambda_0 = \infty$.

\begin{thm}\label{t:noumiyamada}
The $q$RSK algorithm can be reformulated as the following Noumi-Yamada description.

Suppose at time $\ell - 1$ we have a tableau $(\lambda^j_k) = (\lambda^j_k (\ell - 1))$ and want to insert row $w_{\ell, 1 : m}$ into it to obtain a new tableau $(\tilde\lambda^j_k) = (\lambda^j_k (\ell))$. 
We initialise $a^{1 : m}_1 = w_{\ell, 1 : m}$ and recursively apply the following
\begin{align}
  \tilde\lambda^k_k &= \lambda^k_k + a^k_k \label{eq:leftboundary}\\
\tilde\lambda^j_k &= a^j_k + \lambda^j_k + \tilde\lambda^{j - 1}_k -\lambda^{j - 1}_k - \qHyp(\tilde\lambda^{j - 1}_k - \lambda^{j - 1}_k, \lambda^{j - 1}_{k - 1} - \tilde\lambda^{j - 1}_k, \lambda^j_k - \lambda^{j - 1}_k) \qquad j > k \notag\\
a^j_{k + 1} &= a^j_k + \lambda^j_k - \tilde\lambda^j_k + \tilde\lambda^{j - 1}_k - \lambda^{j - 1}_k \notag
\end{align}
\end{thm}

\begin{proof}
  Let us recall the algorithm as described in \cite[Section 6.1 and 6.2]{matveev-petrov15}.
  In natural language it works as follows.
  Suppose we want to insert row $(a_{1 : m})$ into the tableau $(\lambda^j_k)_{1 \le j \le k \le m}$.
  The top particle $\lambda^1_1$ receives a push $a_1$ from the input row and finishes its movement.
  Recursively, when all the particles at level $j - 1$ finishes moving, the increment of the $k$th particle splits into two parts $l^{j - 1}_k$ and $r^{j - 1}_k$, which contribute to the increment of the $k + 1$th and the $k$th  particles at level $j$ respectively.
  The right increment $r^{j - 1}_k$ is a $q$-binomial distributed random variable.
  On top of that the rightmost particle $\lambda^j_1$ of the GT pattern receives a push $a_j$ from the input row.

  To be more precise we present a pseudocode description.

\begin{algorithm}[H]
  \SetKwInOut{Input}{input}\SetKwInOut{Output}{output}
  \Input{tableau $(\lambda^j_k)_{1 \le j \le k \le m}$, row $(a_{1 : m}) \in \mathbb N^n$.}
  \Output{tableau $(\tilde\lambda^k_j)_{1 \le j \le k \le \ell}$.}
\BlankLine
  Initialise $\tilde\lambda^1_1 := \lambda^1_1 + a_1$\;
  \For {$j := 2 : m$}{
    \For {$k := 1 : j$}{
      $\tilde\lambda^j_k := \lambda^j_k$
    }
    $\tilde\lambda^j_1 := \tilde\lambda^j_1 + a_k$\;
    \For {$k := 1 : j - 1$}{
      sample $r^{j - 1}_k \sim \phi_{q^{-1}, q^{\lambda^j_k - \lambda^{j - 1}_k}, q^{\lambda^{j - 1}_{k - 1} - \lambda^{j - 1}_k}}(\cdot | \tilde \lambda^{j - 1}_k - \lambda^{j - 1}_k)$\;
        $\tilde\lambda^j_k := \tilde\lambda^j_k + r^{j - 1}_k$\;
        $l^{j - 1}_k := \tilde\lambda^{j - 1}_k - \lambda^{j - 1}_k - r^{j - 1}_k$\;
        $\tilde\lambda^j_{k + 1} := \tilde\lambda^j_{k + 1} + l^{j - 1}_k$\;
      }
    }
 \caption{$q$RSK}
\end{algorithm}

Next, by matching $a^j_{k + 1}$ as $l^{j - 1}_k$, and noting $\tilde\lambda^{j - 1}_k - \lambda^{j - 1}_k = l^{j - 1}_k + r^{j - 1}_k$ and $\tilde\lambda^j_k = \lambda^j_k + l^{j - 1}_{k - 1} + r^{j - 1}_k$, and rewriting the $q$-binomial distribution as the $q$-hypergeometric distribution using Lemma \ref{l:qbinqhyp} we arrive at the Noumi-Yamada description of the $q$RSK.
\end{proof}

In this article we write $a^j_k(n)$ in place of $a^j_k$ when the insertion is performed at time $n$, namely to transform $(\lambda^j_k(n - 1))$ into $(\lambda^j_k(n))$.

An alternative way of writing down the Noumi-Yamada description is as follows
\begin{equation}
\begin{aligned}
\tilde\lambda^k_k &= \lambda^k_k + a^k_k\\
\tilde\lambda^j_k &= a^j_k + \lambda^j_k + \tilde\lambda^{j - 1}_k -\lambda^{j - 1}_k - a^j_{k + 1} \qquad j > k\\
a^j_{k + 1} &\sim \qHyp(\tilde\lambda^{j - 1}_k - \lambda^{j - 1}_k, \lambda^{j - 1}_{k - 1} - \tilde\lambda^{j - 1}_k, \lambda^j_k - \lambda^{j - 1}_k) 
\end{aligned}
  \label{eq:altny}
\end{equation}

\subsection{Properties of the $q$RSK}
As with the usual RSK, the $q$RSK preserves the interlacing constraints of the GT patterns along levels and time.
\begin{lem}[Lemma 6.2 of \cite{matveev-petrov15}]
  For $\lambda^{j - 1} \prec \lambda^j$ and $\lambda^{j - 1} \prec \tilde\lambda^{j - 1}$, after applying $q$RSK-inserting a row, we have
  \begin{align*}
  \lambda^j \prec \tilde \lambda^j, \qquad \tilde\lambda^{j - 1} \prec \tilde\lambda^j.
  \end{align*}
\end{lem}

\begin{proof}
  By \eqref{eq:qhypsupp} and \eqref{eq:altny}
  \begin{align}
    a^j_k &\le \lambda^j_{k - 1} - \lambda^{j - 1}_{k - 1} \label{eq:l31} \\
    a^j_k &\le \tilde\lambda^{j - 1}_{k - 1} - \lambda^{j - 1}_{k - 1} \label{eq:l32} \\
    a^j_{k + 1} &\le (\tilde\lambda^{j - 1}_k \vee \lambda^j_k ) - \lambda^{j - 1}_k \label{eq:l33}\\
    a^j_{k + 1} &\ge \lambda^j_k - \lambda^{j - 1}_k - \lambda^{j - 1}_{k - 1} + \tilde\lambda^{j - 1}_k \label{eq:l34} 
  \end{align}
  When $j = k$, by \eqref{eq:altny},
  \begin{align*}
    \tilde\lambda^k_k = \lambda^k_k + a^k_k \ge 0
  \end{align*}
  When $j > k$, by \eqref{eq:altny}\eqref{eq:l33}
  \begin{align*}
    \tilde\lambda^j_k \ge a^j_k + (\lambda^j_k \vee \tilde\lambda^{j - 1}_k)
  \end{align*}
  Thus we have shown $\tilde\lambda^j_k \ge \lambda^j_k$ and $\tilde\lambda^j_k \ge \tilde\lambda^{j - 1}_k$.
  By \eqref{eq:altny}\eqref{eq:l31}\eqref{eq:l34} we have
  \begin{align*}
    \tilde\lambda^j_k - \lambda^j_{k - 1} \le \lambda^j_k - \lambda^j_{k - 1} + \tilde\lambda^{j - 1}_k - \lambda^{j - 1}_k + \lambda^j_{k - 1} - \lambda^{j - 1}_{k - 1} - (\lambda^j_k - \lambda^{j - 1}_k - \lambda^{j - 1}_{k - 1} + \tilde\lambda^{j - 1}_k) = 0
  \end{align*}
  Similarly for $j > k$ by \eqref{eq:altny}\eqref{eq:l32}\eqref{eq:l34} we have 
  \begin{align*}
    \tilde\lambda^j_k - \tilde\lambda^{j - 1}_{k - 1} \le \lambda^j_k + \tilde\lambda^{j - 1}_k - \lambda^{j - 1}_k + \tilde\lambda^{j - 1}_{k - 1} - \lambda^{j - 1}_{k - 1} - (\lambda^j_k - \lambda^{j - 1}_k - \lambda^{j - 1}_{k - 1} + \tilde\lambda^{j - 1}_k) = 0
  \end{align*}
  and for $j = k$ by \eqref{eq:altny}\eqref{eq:l32} and that $\lambda^{j - 1} \prec \lambda^j$ we have
  \begin{align*}
    \tilde\lambda^j_k - \tilde\lambda^{j - 1}_{k - 1} \le \lambda^j_k + \tilde\lambda^{j - 1}_{k - 1} - \lambda^{j - 1}_{k - 1} - \tilde\lambda^{j - 1}_{k - 1} = \lambda^j_k - \lambda^{j - 1}_{k - 1} \le 0.
  \end{align*}
\end{proof}

As with the usual RSK, we set the initial tableau to be empty: $\lambda^j_k (0) \equiv 0, \forall k \ge j \ge 1$.

The following lemma shows that the insertion does not ``propagate'' to the $k$th edge before time $k$.
\begin{lem}\label{l:0}
  Starting from the empty initial condition,
  \begin{align*}
    \lambda^j_k (n) = 0 \qquad \forall 0 \le n < k \le j
  \end{align*}
\end{lem}

\begin{proof}
  We show this by induction.

  The empty initial condition is the initial condition for the induction.

  Assume for any $n', j', k'$ such that $0 \le n' < k' \le j', n' \le n, j' \le j, k' \le k, (n', j', k') \neq (n, j, k)$ we have
  \begin{align*}
    \lambda^{j'}_{k'}(n') = 0.
  \end{align*}
  
  If $j > k > 1$, then
  \begin{align*}
    \lambda^j_k (n) &= \lambda^j_k (n - 1) + \lambda^{j - 1}_k (n) - \lambda^{j - 1}_k (n - 1) \\
    &+ \qHyp(\lambda^{j - 1}_{k - 1} (n) - \lambda^{j - 1}_{k - 1} (n - 1), \lambda^{j - 1}_{k - 2}(n - 1) - \lambda^{j - 1}_{k - 1}(n), \lambda^j_{k - 1}(n - 1) - \lambda^{j - 1}_{k - 1} (n - 1)) \\
    &- \qHyp(\lambda^{j - 1}_{k}(n) - \lambda^{j - 1}_k(n - 1), \lambda^{j - 1}_{k - 1} (n - 1) - \lambda^{j - 1}_k (n), \lambda^j_k(n - 1) - \lambda^{j - 1}_k(n - 1))\\
    &= 0
  \end{align*}
  by the induction assumption and \eqref{eq:qhypk=0}.

  The other cases ($j = k$ and $k = 1$) are similar and less complicated.
\end{proof}

The following lemma can be viewed as the boundary case ``dual'' to \eqref{eq:leftboundary}. This duality will become clear when defining the $q$-local moves.
\begin{lem}\label{l:upperboundary}
  With empty initial condition, we have
  \begin{align*}
    \lambda^k_n (n) = 
    \begin{cases}
      \lambda^{k - 1}_n (n) + a^k_n (n) & k > n \\
      a^k_n (n) & k = n
    \end{cases}
  \end{align*}
\end{lem}

\begin{proof}
  When $k > n$, by Theorem \ref{t:noumiyamada}, Lemma \ref{l:0} and \eqref{eq:qhypk=0}
  \begin{align*}
    \lambda^k_n(n) &= \lambda^{k - 1}_n (n) + a^k_n(n) + \lambda^k_n(n - 1) - \lambda^{k - 1}_n (n - 1)\\
    &- \qHyp(\lambda^{k - 1}_n(n) - \lambda^{k - 1}_n (n - 1), \lambda^{k - 1}_{n - 1}(n) - \lambda^{k - 1}_n (n), \lambda^k_n (n - 1) - \lambda^{k - 1}_n (n - 1))\\
    &= \lambda^{k - 1}_n (n) + a^k_n(n).
  \end{align*}

  When $k = n$, by Theorem \ref{t:noumiyamada} and Lemma \ref{l:0}
  \begin{align*}
    \lambda^k_n (n) = \lambda^k_n (n - 1) + a^k_n(n) = a^k_n (n)
  \end{align*}
\end{proof}

The next lemma shows that $q$RSK, like the usual RSK, is weight-preserving.
\begin{lem}\label{l:wt}
  Given empty initial condition, let $(\lambda^k_j (n))$ be the output of $q$RSK taking a matrix $(w_{i, j})$, then almost surely
  \begin{align*}
    \lambda^k_1 (n) + \lambda^k_2 (n) + \dots + \lambda^k_{k \wedge n}(n) = \sum_{i = 1 : n, j = 1 : k} w_{i, j}
  \end{align*}
\end{lem}

\begin{proof}
  We first show by induction that
  \begin{align}
    \sum_{j = 1 : k} \lambda^k_j (n) = \sum_{i = 1 : n, j = 1 : k} w_{i, j} \label{eq:wt}
  \end{align}
  When $n = 0$ the empty initial condition shows that the above formula is true for all $k \ge 1$.

  By recursively applying \eqref{eq:leftboundary} and noting $a^1_1(n) = w_{n, 1}$, we see the above formula is true for $k = 1$ and all $n \ge 0$.

  Assuming the above formula is true for $(k', n') \in \{(k - 1, n), (k, n - 1), (k - 1, n - 1)\}$, summing over $j = 1 : k$ in \eqref{eq:altny}, and by noting $a^k_1(n) = w_{n, k}$ one has
  \begin{align*}
    \sum_{j = 1 : k} \lambda^k_j (n) &= \sum_{j = 1 : k} \lambda^k_j (n - 1) + \sum_{j = 1 : k - 1} \lambda^{k - 1}_j (n) - \sum_{j = 1 : k - 1} \lambda^{k - 1}_j (n - 1) + w_{n, k}\\
    &= \sum_{i = 1 : n - 1, j = 1 : k} w_{i, j} + \sum_{i = 1 : n, j = 1 : k - 1}w_{i, j} - \sum_{i = 1 : n - 1, j = 1 : k - 1} w_{i, j} + w_{n, k} = \sum_{i = 1 : n, j = 1 : k} w_{i, j}.
  \end{align*}
  Then using Lemma \ref{l:0} on \eqref{eq:wt} we arrive at the identity in the statement of the lemma.
\end{proof}

\subsection{The growth diagrams and the symmetry property}
The growth diagram was developed in \cite{fomin86,fomin95}, see also for example \cite[Section 5.2]{sagan00} for an exposition. For the RSK algorithm it is an integer lattice $[n] \times [m]$, where each vertex is labelled by a Young diagram, and each cell labelled by a number.
More specifically, the $(\ell, j)$-cell is labelled by $w_{\ell, j}$, the $(\ell, j)$-th entry of the input matrix, whereas the label of vertex $(\ell, j)$ is the Young diagram $\lambda^j_{1 : j} (\ell)$.

The local growth rule is a function $F_{\RSK}(\lambda, \mu^1, \mu^2, x)$ such that
\begin{align*}
  \lambda^j (\ell) = F_{\RSK} (\lambda^{j - 1} (\ell - 1), \lambda^{j - 1} (\ell), \lambda^j (\ell - 1), w_{j, \ell})
\end{align*}
for all $j$ and $\ell$.
The local growth rule generates the whole diagram.
To see this, one may label the boundary vertices $(0, 0 : m)$ and $(0 : n, 0)$ with the empty Young diagrams, and apply $F_\RSK$ recursively.

By the definition of the $P$- and $Q$-tableaux, the labels of the top row and the labels of the right most column of the growth diagram characterise the $P$- and $Q$-tableaux respectively.
Therefore the symmetry property of the RSK algorithm is reduced to the symmetry property of the local rule:
\begin{align*}
  F_{\RSK} (\lambda, \mu^1, \mu^2, x) = F_{\RSK} (\lambda, \mu^2, \mu^1, x).
\end{align*}
To see this, note that transposing the matrix amounts to transposing the lattice with the cell labels.
By the symmetry property of the local rule, it is invariant under this transposition, therefore one can transpose the vertex labels as well.
As a result the $P$- and $Q$-tableaux are swapped.
This argument will be made more symbolic in the proof of Theorem \ref{t:qsym}.

\subsection{The symmetry property for the $q$RSK}
In the case where the algorithm itself is randomised, or weighted, the local rule branches, and the weights can be placed on the edges.

One example of this is both the column and the row $q$RS defined in \cite{oconnell-pei13,borodin-petrov13} whose symmetry property was proved using this branching version of the growth diagram in \cite{pei14}.

In this section we prove the symmetry property for the $q$RSK in the same way.

\begin{proof}[Proof of Theorem \ref{t:qsym}]
  As in the RSK and $q$RS cases, we first show that the local rule is symmetric. 
  That is, writing $\lambda^m(n) = F_{qRSK}(\lambda^{m - 1} (n - 1), \lambda^{m - 1} (n), \lambda^m (n - 1), w_{m, n})$ then we show $F_{qRSK}(\lambda, \mu^1, \mu^2, x)$ is symmetric in $\mu^1$ and $\mu^2$:
  \begin{align*}
    F_{q\RSK}(\lambda, \mu^1, \mu^2, x) \overset{d}{=} F_{q\RSK}(\lambda, \mu^2, \mu^1, x)
  \end{align*}
  In the rest of the proof we write $F = F_{qRSK}$.

  As remarked before, we use the convention that for any Young diagram $\lambda$, $\lambda_0 = \infty$.
  
  By Theorem \ref{t:noumiyamada} we can write
  \begin{align}
    F(\lambda, \mu^1, \mu^2, x) = (F_k(\lambda, \mu^1, \mu^2, x_k))_{k \ge 1} \label{eq:localrule}
  \end{align}
  where
  \begin{align}
    F_k(\lambda, \mu^1, \mu^2, x_k) = \mu^2_k + \mu^1_k - \lambda_k + x_k - x_{k + 1} \label{eq:localrule2}
  \end{align}
  where $x_1 = x$ and $x_{k + 1} \sim \qHyp(\mu^1_k - \lambda_k, \lambda_{k - 1} - \mu^1_k, \mu^2_k - \lambda_k)$ has pmf symmetric in $\mu^1$ and $\mu^2$:
  \begin{align*}
    f_{x_{k + 1}} (s) &= q^{(\mu^1_k - \lambda_k - s) (\mu^2_k - \lambda_k - s)}\\
    &\times {(\mu^1_k - \lambda_k)_q (\lambda_{k - 1} - \mu^1_k)_q (\mu^2_k - \lambda_k)_q (\lambda_{k - 1} - \mu^2_k)_q \over (s)_q (\mu^1_k - \lambda_k - s)_q (\mu^2_k - \lambda_k - s)_q (\lambda_k + \lambda_{k - 1} - \mu^1_k - \mu^2_k + s)_q (\lambda_{k - 1} - \lambda_k)_q}.
  \end{align*}

  Note that the local rule $F$ does not ``see'' either the level or the time of the insertion.
  Therefore the Young diagrams have to be padded with infinitely many trailing $0$s.
  This is why we let the edge index $k$ range from $1$ to $\infty$ in \eqref{eq:localrule}.
  It is consistent with the Noumi-Yamada description in the boundary case $j = k$ and the ``null'' case $j < k$.
  When $j = k$,
  \begin{align*}
    \tilde \lambda^k_k = a^k_k + \lambda^k_k + \tilde \lambda^{k - 1}_k - \lambda^{k - 1}_k - \qHyp(\tilde\lambda^{k - 1}_k - \lambda^{k - 1}_k, \lambda^{k - 1}_{k - 1} - \tilde\lambda^{k - 1}_k, \lambda^k_k - \lambda^{k - 1}_k) = \lambda^k_k + a^k_k
  \end{align*}
  due to $\tilde\lambda^{k - 1}_k = \lambda^{k - 1}_k = 0$ and \eqref{eq:qhypk=0}.
  Similarly when $j < k$ all terms in the right hand side of \eqref{eq:localrule2} are $0$, so that $\tilde\lambda^j_k$ can stay $0$.

  The rest follows the same argument as in the proof of \cite[Theorem 3]{pei14}.
  Here we produce a less visual and more symbolic argument.

  Let $\pi = ((0, m) = (j_0, k_0) \to (j_1, k_1) \to \dots \to (j_{m + n}, k_{m + n}) = (n, 0))$ be a down-right path from $(0, m)$ to $(n, 0)$, that is $(j_i, k_i) - (j_{i - 1}, k_{i - 1}) \in \{(0, -1), (1, 0)\}$. 
  Let $G$ be the enclosed area of $(j, k)$:
  \begin{align*}
    G = \{(j', k'): \exists i \text{ such that } 0 \le j' \le j_i, 0 \le k' \le k_i\}
  \end{align*}
  Let $A(G) = \{w_{i, j}: i, j \ge 1, (i, j) \in G\}$ be the weights in the $G$-area.
  Denote by $L(A, \pi) = (\lambda^{k_i}(j_i))_i$ the Young diagrams labelled along $\pi$ with input $A$.
  It suffices to show that for any $(A, \pi)$ the $L$ satisfies a symmetry property:
  \begin{align}
    L(A, \pi) \overset{d}{=} L(A', \pi')^r \label{eq:sympath}
  \end{align}
  where $\pi' := ((k_{m + n}, j_{m + n}), (k_{m + n - 1}, j_{m + n - 1}), \dots, (k_1, j_1))$ is the transpose of $\pi$, and $\cdot^r$ is the reverse: $b^r := (b_\ell, b_{\ell - 1}, \dots, b_1)$ for a tuple $b = (b_1, b_2, \dots, b_\ell)$.

  The symmetry property of the follows when one takes $\pi = \pi^* = ((0, m) \to (1, m) \to \dots \to (n, m) \to (n, m - 1) \to \dots \to (n, 0))$.

  We show \eqref{eq:sympath} by induction.
  When $\pi = ((0, m) \to (0, m - 1) \to \dots \to (0, 0) \to (1, 0) \to \dots \to (n, 0))$ it is true because by the boundary condition all the diagrams along $\pi$ are empty.

  When \eqref{eq:sympath} is true for path $\pi$, assume $\pi \neq \pi^*$ (otherwise we are done), then there exists at least one $\ell$ such that $j_\ell + 1= j_{\ell + 1}$ and $k_\ell + 1 = k_{\ell - 1}$, that is $(j_{\ell - 1}, k_{\ell - 1}), (j_\ell, k_\ell), (j_{\ell + 1}, k_{\ell + 1})$ forms an L-shape.

  Now one can apply the symmetry of the local growth rule $F$ to the cell containing this L-shape, to obtain \eqref{eq:sympath} for $\pi^+$, where $\pi^+$ has the same coordinates as $\pi$ except that $(j_\ell, k_\ell) := (j_\ell + 1, k_\ell + 1)$.
\end{proof}

\section{$q$-polymer}\label{s:qpolymer}
\subsection{From RSK algorithms to polymers}
For the RSK algorithm, due to the Greene's Theorem \cite{greene74} the first edge of the output tableaux are the partition function of the directed last passage percolation (DLPP) of the input matrix:
Let $(\lambda^j_k(\ell))$ be the output of the RSK algorithm taking matrix $(w_{i, j})_{n \times m}$, then
\begin{align*}
  Z_0(\ell, j) := \lambda^j_1(\ell) = \max_{\pi: (1, 1) \to (\ell, j)} \sum_{(i, j) \in \pi} w_{i, j},
\end{align*}
where $\pi: (1, 1) \to (\ell, j)$ indicates $\pi$ is an upright path from $(1, 1)$ to $(\ell, j)$.

Locally, $Z_0$ satisfies the following recursive relation, which is what happens at the first edge in the Noumi-Yamada description:
\begin{align*}
  Z_0 (\ell, j) = (Z_0(\ell, j - 1) \vee Z_0(\ell - 1, j)) + w_{\ell, j}.
\end{align*}

Similarly for the gRSK, the first edge corresponds to the partition functions of the directed polymer (DP) of the matrix:
\begin{align*}
  Z_1(\ell, j) := \log z^j_1(\ell) = \log \left(\sum_{\pi: (1, 1) \to (\ell, j)} \prod_{(i, j) \in \pi} e^{w_{i, j}}\right)
\end{align*}

And locally we have
\begin{align*}
  Z_1(\ell, j) = \log(e^{Z_1(\ell, j - 1)} + e^{Z_1(\ell - 1, j)}) + w_{\ell, j}.
\end{align*}

Because of this, we define the $q$-polymer by focusing on the first edge $Z(n, m) := \lambda^m_1(n)$.
\footnote{Note that the first edge of the $q$RSK was regarded as an interacting particle system called the geometric $q$-pushTASEP in~\cite[Section 6.3]{matveev-petrov15}, which we will also consider in the next section.}
Then by the Noumi-Yamada description of the $q$RSK the $q$-polymer can be defined locally by
\begin{align*}
  Z_q(1, 1) &= w_{1, 1}, \\
  Z_q(n, 1) &= Z_q(n - 1, 1) + w_{n, 1}, \qquad n > 1\\
  Z_q(1, m) &= Z_q(1, m - 1) + w_{1, m}, \qquad m > 1\\
  Z_q(n, m) &= w_{n, m} + Z_q(n - 1, m) + Z_q(n, m - 1) - Z_q(n - 1, m - 1) - X' \\
  X' &\sim \qHyp(Z_q(n, m - 1) - Z_q(n - 1, m - 1), \infty, Z_q(n - 1, m) - Z_q(n - 1, m - 1)),\; m, n > 1
\end{align*}

It is not known whether a more global interpretation of $Z_q$ for $0 < q < 1$ exists, like the first definitions of $Z_0$ and $Z_1$ involving directed paths. 
More generally, the full Greene's theorem interpretes the sum of the first $k$ edges of a fixed level of the (g)RSK-output triangular patterns as similar statistics of $k$ directed non-intersecting paths, but the $q$ version of this theorem is also unknown, so is a sensible definition of the $q$ version of $k$-polymers.

But locally, the DLPP, DP and $q$-polymer models are very similar, as we shall see now.

\subsection{$q$-Burke property}\label{s:qburke}
Fix $\ell$ and $j$, let
\begin{align*}
  U_q &= Z_q(\ell, j - 1) - Z_q(\ell - 1, j - 1) \\
  V_q &= Z_q(\ell - 1, j) - Z_q(\ell - 1, j - 1) \\
  X_q &= w_{\ell, j} \\
  U_q' &= Z_q(\ell, j) - Z_q(\ell - 1, j) \\
  V_q' &= Z_q(\ell, j) - Z_q(\ell, j - 1)
\end{align*}

\begin{lem}
  For $0 \le q \le 1$ we have
  \begin{align*}
    U_q' - U_q = V_q' - V_q = X_q - X_q' \qquad (\text{B1}.q)
  \end{align*}
  where
  \begin{align*}
    X_q' = X_q'(U_q, V_q)
    \begin{cases}
      = U_0 \wedge V_0 & q = 0 \qquad (\text{B2}.0)\\
      \sim \qHyp(U_q, \infty, V_q) & 0 < q < 1 \qquad (\text{B2}.q)\\
      = - \log (e^{- U_1} + e^{- V_1}) & q = 1  \qquad (\text{B2}.1)  
    \end{cases}
  \end{align*}
\end{lem}
\begin{proof}
  Immediate from the Noumi-Yamada descriptions at the first edge.
\end{proof}
We call (B1.$q$) (B2.$q$) the Burke relations. When $q = 0$ or $1$, the Burke relations define the RSK algorithms because the dynamics are the same along all edges of the GT patterns, whereas when $q \in (0, 1)$ the $q$RSK dynamics in the non-first edges are different from the Burke relation.

Also for $q = 0$ or $1$, when $U_q$, $V_q$ and $X_q$ are random with certain distributions, the Burke relations yield the Burke properties in the DLPP and DP cases.

Let us recall the Burke properties in these two cases. 
For convenience we omit the subscripts.

\begin{prp}\label{p:burkegeom}
  Let $(U, V, X, U', V', X')$ satisfy the Burke relations (B1.$q$) and (B2.$q$).
  Suppose $(U, V, X)$ are independent random variables with one of the following distributions
  \begin{itemize}
    \item When $q = 0$ 
      \begin{itemize}
        \item Fix $0 < \alpha, \beta < 1$. Suppose $U \sim Geom(1 - \alpha)$, $V \sim Geom(1 - \beta)$ and $X \sim Geom(1 - \alpha \beta)$.
        \item Or fix $\alpha, \beta > 0$. Suppose $U \sim Exp(\alpha)$, $V \sim Exp(\beta)$ and $X \sim Exp(\alpha + \beta)$.
      \end{itemize}
    \item When $q = 1$, fix $\alpha, \beta > 0$. Suppose $\exp(-U) \sim Gamma(\alpha)$, $\exp(-V) \sim Gamma(\beta)$ and $\exp(-X) \sim Gamma(\alpha + \beta)$.
  \end{itemize}
  Then in each of the above cases
  \begin{align*}
    (U', V', X') \overset{d}{=} (U, V, X).
  \end{align*}
\end{prp}
The Burke property with geometric weights can be found in e.g. \cite[Lemma 2.3]{seppalainen09}, the one with the exponential weights in \cite{balazs-cator-seppalainen06}, the one with log-gamma weights in \cite{seppalainen12}.

The $q$-Burke property is similar.
\begin{prp}\label{t:qburke}
  Let $(U, V, X, U', V', X')$ satisfy (B1.$q$) and (B2.$q$) with $0 < q < 1$. Let $0 < \alpha, \beta < 1$. Let $U, V$ and $X$ be independent random variables such that $U \sim \qGeom(\alpha)$, $V \sim \qGeom(\beta)$ and $X \sim \qGeom (\alpha \beta)$. 
Then $(U', V', X') \overset{d}{=} (U, V, X)$.
\end{prp}

\begin{proof} By the definitions of the $q$-geometric and the $q$-hypergeometric distributions,
  \begin{align*}
    \prob(U' &= u, V' = v, X' = x) \\
    &= \prob(U + X = u + x, V + X = v + x, X' = x) \\
    &= \sum_y \prob(X = y, U = u + x - y, V = v + x - y, X' = x)\\
    &= \sum_y {(\alpha\beta)^y \over (y)_q} (\alpha\beta; q)_\infty {\alpha^{u + x - y} \over (u + x - y)_q} (\alpha; q)_\infty {\beta^{v + x - y} \over (v + x - y)_q} (\beta; q)_\infty \\
    &\;\;\;\;\;\;\;\times q^{(u - y)(v - y)} {(u + x - y)_q (v + x - y)_q \over (x)_q (u - y)_q (v - y)_q} \\
    &= (\alpha\beta; q)_\infty (\alpha; q)_\infty (\beta; q)_\infty {(\alpha\beta)^x \over (x)_q} \alpha^{u} \beta^{v} \sum_y {q^{(u - y) (v - y)} \over (y)_q (u - y)_q (v - y)_q} \\
    &= (\alpha\beta; q)_\infty (\alpha; q)_\infty (\beta; q)_\infty {(\alpha\beta)^x \over (x)_q} {\alpha^{u} \over (u)_q} {\beta^{v} \over (v)_q}
  \end{align*}
  where the last identity is due to \eqref{eq:qhypinf}.
\end{proof}

When $q = 0$ or $1$ the converse of Proposition \ref{p:burkegeom} is also true (see e.g. \cite{seppalainen12} for the $q = 1$ case). 
That is, the Burke relation and the indentification in law of the triplets implies the specific distributions (geometric, exponential and loggamma) under reasonable assumptions thanks to the characterisation results in~\cite{ferguson64,ferguson65,lukacs55}.
The converse of the $q$-Burke property is open.

The $q$-Burke property allows one to tackle the $q$-polymer on the $\pint^2$ lattice (obtained by a simple shift of the model on the $\pint_{>0}^2$ lattice in previous considerations) with the following condition:
\begin{align*}
  w_{0, 0} &= 0 \\
  w_{i, 0} &\sim \qGeom(\alpha), \qquad i \ge 1 \\
  w_{0, j} &\sim \qGeom(\beta), \qquad j \ge 1 \\
  w_{i, j} &\sim \qGeom(\alpha \beta) \qquad i, j \ge 1
\end{align*}
We call such a configuration a $q$-polymer with stationary boundary conditions.

Now we can show the strong law of large numbers of the partition functions.
\begin{proof}[Proof of Theorem \ref{t:lln}]
  The proof is similar to the version of DLLP with geometric weights, see e.g. \cite[Theorem 4.12]{romik14}.

  Let us consider the increment of $Z$ along the paths $(0, 0) \to (1, 0) \to \dots \to (\ell, 0) \to (\ell, 1) \to (\ell, 2) \to \dots \to (\ell, j)$.
  Let 
  \begin{align*}
    U(k) &= Z(k, 0) - Z(k - 1, 0), \qquad k = 1 : \ell \\
    V(k') &= Z(\ell, k') - Z(\ell, k' - 1), \qquad k' = 1 : j
  \end{align*}
  The horizontal increment $U(1 : \ell) = w_{1 : \ell, 0}$ are i.i.d. random variables with distribution $\qGeom(\alpha)$.
  And by using Proposition \ref{t:qburke} recursively, we have that the vertical increments $V(1 : j)$ are i.i.d. random variable with distribution $\qGeom(\beta)$.

  Using \eqref{eq:qgeommoment} we obtain \eqref{eq:qpolyexp}, and with the usual strong law of large numbers we obtain \eqref{eq:lln}.
\end{proof}

In \cite[Section 6.3]{matveev-petrov15}, the dynamics of the first edge of the tableaux was formulated as an interacting particle system, called the geometric $q$-pushTASEP.
Therefore it has a natural correspondence with the $q$-polymer, where $Z(n, m) + m = \xi_m(n)$ is the location of the $m$th particle at time $n$.

Here we describe the geometric $q$-pushTASEP whose initial condition corresponds to the $q$-polymer with stationary boundary condition.

\begin{dfn}[Stationary geometric $q$-pushTASEP]
  Let $(\xi_0 (n), \xi_1 (n), \dots)$ be the locations of the particles at time $n$, such that $\xi_0(n) \le \xi_1(n) \le \dots$ for all $n$.
  Initially, $\xi_0(0) = 0$, $\xi_m(0) - \xi_{m - 1}(0) - 1 \sim \qGeom(\beta)$.
  That is, the $0$th particle is at $0$, and the gaps between consecutive particles are independently $q$-geometric distributed random variables with parameter $\beta$.
  At time $n$, the $0$th particle jumps forward by a distance distributed according to $q$-geometric distribution with paramter $\alpha$, and sequentially given that the $m - 1$th particle has jumped, the $m$th particle jumps forward by a distance as a sum of a $q$-geometric random variable with paramter $\alpha \beta$ and a random variable $Y$ distributed according to $\xi_{m - 1} (n) - \xi_{m - 1} (n - 1) - 1 - Y \sim \qHyp(\xi_{m - 1} (n) - \xi_{m - 1}(n - 1), \infty, \xi_m (n - 1) - \xi_{m - 1}(n - 1) - 1)$.
\end{dfn}

Thus via the translation of (the arguments in the proof of) Theorem \ref{t:lln} we have
\begin{cly}
  Let $\xi_{0 : \infty}$ be the locations of the stationary geometric $q$-pushTASEP.
  Then we have the following
  \begin{enumerate}
    \item For any $j \ge 0$, the $j$th particle performs a simple random walk with increments distributed according to $\qGeom(\alpha)$.
      \item At any time, the gap between neighbouring particles are independently distributed according to $\qGeom(\beta)$.
      \item Almost surely, $\lim_{N \to \infty} {\xi_{\lfloor N y \rfloor} (\lfloor N x \rfloor)} / N = x \gamma(\alpha) + y (\gamma(\beta) + 1)$ for $x, y \ge 0$.
  \end{enumerate}
\end{cly}

\subsection{Classical limit of the Burke relations}
It is natural to guess that under the classical limit of $q$-Burke relation (B1.$q$) and (B2.$q$) becomes the Burke relation (B1.1) and (B2.1) of the DP.
Here we give a heuristic argument justifying this statement.
The argument may be compared to that in the proof of \cite[Lemma 8.17]{matveev-petrov15}.

In the rest of this section, for convenience we omit the $\lfloor \cdot \rfloor$ when an integer is required.
Given $U, V$ we define
\begin{align*}
  U_\epsilon &= m(\epsilon) + \epsilon^{-1} U \\
  V_\epsilon &= m(\epsilon) + \epsilon^{-1} V
\end{align*}
and sample $X_\epsilon \sim \qHyp(U_\epsilon, \infty, V_\epsilon)$, then by Items 2 and 3 of Lemma \ref{l:qclassicallimits}
\begin{align*}
  \prob(X_\epsilon &= m(\epsilon) + \epsilon^{-1} x) = q^{\epsilon^{- 2} (U - x) (V - x)} {(m(\epsilon) + \epsilon^{-1} U)_q (m(\epsilon) + \epsilon^{-1} V)_q \over (m(\epsilon) + \epsilon^{-1} x)_q (\epsilon^{-1}(U - x))_q (\epsilon^{-1} (V - x))_q}\\
  &= \exp( \epsilon^{-1}(- (U - x) (V - x) + {\pi^2 \over 6} - \Li_2(e^{x - U}) - \Li_2(e^{x - V})) + o(\epsilon^{-1})) \\
  &=: \exp(\epsilon^{-1} f(x) + o(\epsilon^{-1}))
\end{align*}

Using the reflection property of the dilogarithm function
\begin{align*}
  \Li_2(z) + \Li_2(1 - z) = {\pi^2 \over 6} - \log z \log (1 - z)
\end{align*}
we have
\begin{align*}
  f(- \log (e^{- U} + e^{- V})) = 0.
\end{align*}

By taking derivatives of $f$ we also have
\begin{align*}
  f'(- \log (e^{- U} + e^{- V})) = 0\\
  f''(x) < 0
\end{align*}
Hence $f$ achieves unique maximum $0$ at $- \log (e^{- U} + e^{- V})$.

Now we can define $X(\epsilon)$ by the relation $X_\epsilon = m(\epsilon) + \epsilon^{-1} X(\epsilon)$ and obtain 
\begin{align*}
  X(\epsilon) \to X' = - \log (e^{- U} + e^{- V})
\end{align*}
and we have recovered (B2.1).

\section{$q$-local moves}\label{s:qlocalmoves}
In this section we define the $q$-local moves and prove Theorem \ref{t:lmpush}.

In a sense, the local moves are very fundamental building blocks, as they unify the PNG and the RSK.

Let us define an object by adding to a 2 by 2 matrix a labelled edge connecting the $21$- and $22$-entries:
$
\begin{pmatrix}
  a & & b \\ c & \underset{e}{\text{---}} & d
\end{pmatrix}
$

The $q$-deformation of local moves consist of two maps:
\begin{align*}
  l:
  \begin{pmatrix}
    a & & b \\ c & \underset{e}{\text{---}} & d
  \end{pmatrix}
  \mapsto
  \begin{pmatrix}
    a' & b \\ c & b + c + d - a - a'
  \end{pmatrix};
  \qquad
  l':
  \begin{pmatrix}
    a & b \\ c & d
  \end{pmatrix}
  \mapsto
  \begin{pmatrix}
    a & & b \\ c & \underset{d - c}{\text{---}} & d
  \end{pmatrix}
\end{align*}
where $a'$ is a random variable with $q$-hypergeometric distribution $\qHyp(c - a, e, b - a)$.

On the boundary we define
\begin{align}
  l: 
  \begin{pmatrix}
      & &   \\ c & \underset{-}{\text{---}} & d
  \end{pmatrix}
  &\mapsto
  \begin{pmatrix}
     &  \\ c & c + d
   \end{pmatrix}; \label{eq:lmb1}\\
  \begin{pmatrix}
    & & b \\ & \underset{-}{\text{---}} & d
  \end{pmatrix}
  &\mapsto
  \begin{pmatrix}
    & b \\ & b + d
  \end{pmatrix}; \label{eq:lmb2}\\
  \begin{pmatrix}
    && \\ & \underset{-}{\text{---}}  & d
  \end{pmatrix}
  &\mapsto
  \begin{pmatrix}
    & \\ & d
  \end{pmatrix}. \label{eq:lmb3}
\end{align}
And
\begin{align*}
  l': 
  \begin{pmatrix}
     &  \\ c & d
  \end{pmatrix}
  &\mapsto
  \begin{pmatrix}
      & &   \\ c & \underset{-}{\text{---}} & d
  \end{pmatrix};\\
  \begin{pmatrix}
    & b \\ & d
  \end{pmatrix} 
  &\mapsto
  \begin{pmatrix}
    & & b \\ & \underset{-}{\text{---}} & d
  \end{pmatrix}; \\
  \begin{pmatrix}
    & \\ & d
  \end{pmatrix}
  &\mapsto
  \begin{pmatrix}
    && \\ & \underset{-}{\text{---}}  & d
  \end{pmatrix}.
\end{align*}

Given an array $A = (w_{i j})_{i, j \ge 1}$ with labelled horizontal edges connecting neighbouring entries in the same rows, let $l_{ij}$ and $l'_{ij}$ be $l$ and $l'$ acting on the $(i, j)$-th sub-2 by 2 matrix, namely $l_{i, j}$ (resp. $l'_{i, j}$) acts on $A$ by acting on
$ \begin{pmatrix} w_{i - 1, j - 1} & & w_{i - 1, j} \\ w_{i, j - 1} & \underset{e_{i j}}{\text{---}} & w_{i, j} \end{pmatrix} $ 
\Bigg(resp. 
$ \begin{pmatrix} w_{i - 1, j - 1} & w_{i - 1, j} \\ w_{i, j - 1} & w_{i, j} \end{pmatrix} $
\Bigg)
and keeping other entries unchanged.

Similarly as in \cite{oconnell-seppalainen-zygouras14}, define $\rho_{i j}$ by
\begin{align*}
  \rho_{i j} &= (l'_{(i - 1 - j)^+ + 1, (j - i + 1)^+ + 1} \circ \dots \circ l'_{i - 2, j - 1} \circ l'_{i - 1, j} )\\
  &\qquad\qquad\circ (l_{(i - j)^+ + 1, (j - i)^+ + 1} \circ \dots \circ l_{i - 1, j - 1} \circ  l_{i j})
\end{align*}
where for any integer $n$ we denote $(n)^+ := n \vee 0$ to be the positive part of $n$.
The operator $l'_{ij}$ are purely auxiliary, as they only serve to store the differences like $t_{i, j} - t_{i, j - 1} = \lambda^{j - 1}_{k - 1} - \tilde\lambda^{j - 1}_{k}$ before $t_{i, j}$ is unrecoverably changed (see the proof of Theorem \ref{t:qlocalmoves} for more details).

Given an input array $(w_{i j})$, we initialise by labelling all the horizontal edges between $w_{i - 1, j}$ and $w_{i, j}$ with $e_{i, j} = \infty$. 

For two paritions $\lambda$ and $\mu$, denote by $\lambda \nearrow \mu$ if $\lambda \prec \mu$ and $\mu = \lambda + \ve_i$ for some $i$.

Let $\Lambda$ be a Young diagram of size $N$, and $\emptyset = \Lambda(0) \nearrow \Lambda(1) \nearrow \Lambda(2) \nearrow \dots \nearrow \Lambda(N) = \Lambda$ be a sequence of growing Young diagrams, which we call a growth sequence of $\Lambda$.
For $\lambda \nearrow \mu$, denote $\mu / \lambda$ as the coordinate of the box added to $\lambda$ to form $\mu$.
For example, if $\lambda = (4, 2, 1)$ and $\mu = (4, 3, 1)$ then $\mu / \lambda = (2, 3)$.
As aside, it is well known that a growth sequence $\Lambda(0 : N)$ of $\Lambda$ corresponds to a standard Young tableau $T$ of shape $\Lambda$, where $T$ can be obtained by filling the box with coordinate ${\Lambda(i) / \Lambda(i - 1)}$ by $i$.
Now define
\begin{align*}
  T_\Lambda = \rho_{\Lambda / \Lambda(N - 1)} \circ \dots \circ \rho_{\Lambda(2) / \Lambda(1)} \circ \rho_{\Lambda(1) / \emptyset}
\end{align*}
to be an operator acting on integer arrays on $\pint_{>0}^2$.
It does not depend on the choice of the sequence as we shall see in the proof of Theorem \ref{t:qlocalmoves}, hence it is well defined.

Denote by $S(\Lambda)$ the boundary of $\Lambda$:
\begin{align*}
  S(\Lambda) = \{(i, j) \in \Lambda: (i + 1, j + 1) \not\in \Lambda\}
\end{align*}
The set $S(\Lambda)$ determines a coordinate system of all cells in $\Lambda$.
To see this, for any $(i', j') \in \Lambda$, there exists unique $(i, j) \in S(\Lambda)$ and unique $k \ge 1$ such that $(i', j')$ and $(i, j)$ are at the same diagonal and their ``distance'' is $k - 1$:
\begin{align*}
  i - i' = j - j' = k - 1
\end{align*}
In this case we call $(i, j, k)$ the $\Lambda$-coordinate of $(i', j')$.

In the following, for some big integers $\hat N$ and $\hat M$, let $I(\Lambda) := [\hat N] \times [\hat M]$ be a rectangular lattice covering $\Lambda$: $I(\Lambda) \supset \Lambda$.

\begin{thm}\label{t:qlocalmoves}
  Let $(t_{ij}) = T_\Lambda A$.
  For any $(i', j') \in \Lambda$ with $\Lambda$-coordinate $(i, j, k)$
  \begin{align*}
    t_{i', j'} = \lambda^j_k(i).
  \end{align*}
  where $(\lambda^j_k(i))$ is the output of $q\RSK(A(I(\Lambda)))$. Note the above equality is an identity in joint distribution over all boxes $(i', j')$.

  Specifically when $\Lambda = [n] \times [m]$ is a rectangular lattice, by specifying $I(\Lambda) = \Lambda$ the $P$- and $Q$-tableaux of $q\RSK(A(\Lambda))$
\begin{align*}
  \lambda^k_j &= t_{n - j + 1, k - j + 1}, \qquad j = 1 : k \wedge n, k = 1 : m \\
  \mu^k_j &= t_{k - j + 1, m - j + 1}, \qquad k = 1 : n, j = 1 : k \wedge m
\end{align*}
form exactly the output matrix $T_\Lambda A(\Lambda)$, thus the local moves coincide with the $q$RSK algorithm taking the matrix $A(\Lambda)$ in this case.
\end{thm}
Here is an illustration, where we show the shape $\Lambda$, and $t_{i'j'}$ which corresponds to $\lambda^i_k(j)$.
\begin{center}
  \includegraphics{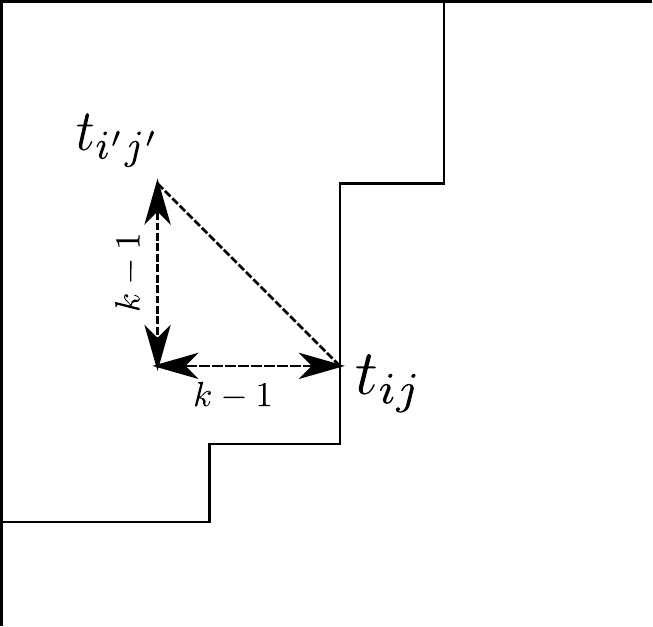}
\end{center}

\begin{proof}
  We prove it by induction.

  When $\Lambda = (1)$, that is, it is a one-by-one matrix, applying the local move $T_\Lambda = \rho_{1, 1} = l_{1, 1}$ on $A$ we obtain the correct result $\lambda^1_1 (1) = w_{1, 1}$.

  Let $\Theta \nearrow \Xi$ be two Young diagrams such that $\Xi / \Theta = (n, k)$.
  We assume the theorem is true for $\Lambda = \Theta$ and want to show it holds for $\Lambda = \Xi$.

  For all $(i, j)$ with $i - j \neq n - k$, on the one hand, the $\Theta$-coordinate and the $\Xi$-coordinate of $(i, j)$ coincide, and on the other hand,
  \begin{align*}
    T_\Xi A(i, j) = \rho_{n, k} T_\Theta A(i, j) = T_\Theta A (i, j) 
  \end{align*}
  as $\rho_{n, k}$, by its definition, only alters the entries along the diagonal $i - j = n - k$.
  Therefore it suffices to show
  \begin{align*}
    (\rho_{n, k} T_\Theta A) (n - \ell + 1, k - \ell + 1) = \lambda^k_\ell(n), \qquad \ell = 1 : n \wedge k.
  \end{align*}

  Once again we use an induction argument.
  Denote for $\ell = 0 : n \wedge k$
  \begin{align*}
    t^\ell = l_{n - \ell + 1, k - \ell + 1} \circ l_{n - \ell, k - \ell} \circ \dots \circ l_{n - 1, k - 1} \circ l_{n, k} T_\Theta A.
  \end{align*}
  Then $t^0 = T_\Theta A$ and $t^{n \wedge k} = T_\Xi A$.

  It suffices to show that for $\ell = 1 : n \wedge k - 1$
  \begin{align*}
    t^\ell_{n - i + 1, k - i + 1} =
    \begin{cases}
      \lambda^k_i(n) & 1 \le i \le \ell \\
      a^k_{\ell + 1} (n) & i = \ell + 1 \\
      \lambda^{k - 1}_{i - 1} (n - 1) & \ell + 2 \le i \le n \wedge k
    \end{cases},
  \end{align*}
  and for $\ell = n \wedge k$, $t^\ell_{n - i + 1, k - i + 1} = \lambda^k_i(n)\forall i = 1 : n \wedge k$.

  We consider the bulk case, that is when $n, k > 1$, as the boundary cases are similar and much easier.

  For $\ell = 1$, when $l_{n, k}$ acts on $t^0$, by the Noumi-Yamada description it alters the submatrix (note that $w_{n, k} = a^k_1(n)$)
  \begin{align*}
    \begin{pmatrix}
      t^0_{n - 1, k - 1} & & t^0_{n - 1, k} \\ t^0_{n, k - 1} & \underset{\infty}{\text{---}} & t^0_{n, k}
    \end{pmatrix}
    =
    \begin{pmatrix}
      \lambda^{k - 1}_1 (n - 1) & \lambda^{k}_1 (n - 1) \\ \lambda^{k - 1}_1 (n) & w_{n, k}
    \end{pmatrix}
  \end{align*}
  into
  \begin{align*}
    \begin{pmatrix}
      a^k_2(n) & \lambda^k_1 (n - 1) \\ \lambda^{k - 1}_1 (n) & \lambda^k_1(n)
    \end{pmatrix}
  \end{align*}

  For $1 < \ell < n \wedge k$, given the induction assumption, $l_{n - \ell + 1, k - \ell + 1}$ acts on $t^{\ell - 1}$ by changing
  \begin{align*}
    \vmat{t^{\ell - 1}_{n - \ell, k - \ell}}{t^{\ell - 1}_{n - \ell, k - \ell + 1}}{t^{\ell - 1}_{n - \ell + 1, k - \ell}}{t^{\ell - 1}_{n - \ell + 1, k - \ell + 1}}{e}
    = \vmat{\lambda^{k - 1}_\ell (n - 1)}{\lambda^k_\ell (n - 1)}{\lambda^{k - 1}_\ell (n)}{a^k_\ell(n)}{\lambda^{k - 1}_{\ell - 1} (n - 1) - \lambda^{k - 1}_\ell (n)}
  \end{align*}
into
  \begin{align*}
    \begin{pmatrix}
      a^k_{\ell + 1}(n) & \lambda^k_\ell (n - 1) \\ \lambda^{k - 1}_\ell(n) & \lambda^k_\ell (n)
    \end{pmatrix}
  \end{align*}
  and that $l'_{n - \ell, k - \ell + 1}$ transforms the submatrix
  \begin{align*}
    \begin{pmatrix}
      t^\ell_{n - \ell - 1, k - \ell} & t^{\ell}_{n - \ell - 1, k - \ell + 1} \\ t^\ell_{n - \ell, k - \ell} & t^\ell_{n - \ell, k - \ell + 1}
    \end{pmatrix}
    =
    \begin{pmatrix}
      \lambda^{k}_{\ell + 1} (n - 1) & \lambda^{k + 1}_{\ell + 1} (n - 1) \\ \lambda^{k}_{\ell + 1} (n) & \lambda^{k}_{\ell} (n - 1)
    \end{pmatrix}
  \end{align*}
  into
  \begin{align*}
    \vmat{\lambda^{k}_{\ell + 1}(n - 1)}{\lambda^{k + 1}_{\ell + 1} (n - 1)}{\lambda^{k}_{\ell + 1} (n)}{\lambda^{k}_{\ell} (n - 1)}{\lambda^{k}_{\ell}(n - 1) - \lambda^{k}_{\ell + 1} (n)}
  \end{align*}
  which stores the correct argument for a possible future operation $\rho_{n, k + 1}$.

  For $\ell = n \wedge k$, say $n > k$, then by the induction assumption and the definition of the local moves at the left boundary \eqref{eq:lmb2}, $l_{n - k + 1, 1}$ acts on $t^{k - 1}$ by changing
  \begin{align*}
      \begin{pmatrix}
        & & t^{k - 1}_{n - k, 1} \\ & \underset{-}{\text{---}} & t^{k - 1}_{n - k + 1, 1}
  \end{pmatrix}
  =
      \begin{pmatrix}
        & & \lambda^{k}_{k}(n - 1) \\ & \underset{-}{\text{---}} & a^k_k(n)
  \end{pmatrix}
  \end{align*}
  into
  \begin{align*}
  \begin{pmatrix}
    & \lambda^k_k(n - 1) \\ & \lambda^k_k(n)
  \end{pmatrix}.
  \end{align*}
  This is the boundary case \eqref{eq:leftboundary} of the Noumi-Yamada description.

  Similarly when $n = k$ and $n < k$, the upper boundary and upper-left boundary cases \eqref{eq:lmb1}\eqref{eq:lmb3} correspond to Lemma \ref{l:upperboundary}.
\end{proof}

We also note a $q$-analogue of the map $b_{i,j}$ in \cite[(3.5)]{oconnell-seppalainen-zygouras14} (or the octahedron recurrence as in \cite{hopkins14}). 
Applying $\rho_{n, k}$ to a tri-diagonal strip $(i, j)_{n - k - 1 \le i - j \le n - k + 1}$, in the bulk, that is when $i - j = n - k, i, j > 1, i < n$ we have
\begin{align}
  t_{i, j} &:= t_{i - 1, j} + t_{i, j - 1} - t_{i - 1, j - 1} + \qHyp(t_{i + 1, j} - t_{i, j}, t_{i + 1, j + 1} - t_{i + 1, j}, t_{i, j + 1} - t_{i, j}) \notag \\
  &\qquad\qquad- \qHyp(t_{i, j - 1} - t_{i - 1, j - 1}, t_{i, j} - t_{i, j - 1}, t_{i - 1, j} - t_{i - 1, j - 1}), \qquad i < n - 1 \label{eq:t1}\\
  t_{i, j} &:= t_{i - 1, j} + t_{i, j - 1} - t_{i - 1, j - 1} + \qHyp(t_{i + 1, j} - t_{i, j}, \infty, t_{i, j + 1} - t_{i, j}) \notag\\
  &\qquad\qquad- \qHyp(t_{i, j - 1} - t_{i - 1, j - 1}, t_{i, j} - t_{i, j - 1}, t_{i - 1, j} - t_{i - 1, j - 1}), \qquad i = n - 1 \label{eq:t2}
\end{align}
where all the $q$-hypergeometric random variables with distinct parameters are independent.

\subsection{The push-forward measure of the $q$-local moves}
In this section we prove Theorem \ref{t:lmpush}.
Before starting the proof, we show some illustrations to help with the readability.

Here is an illustration of the measure $\mu_{q, \Lambda}$ for $\Lambda = (5, 5, 3, 3, 1)$.
Some of the $t$-entries have been labelled.
We focus on the products of $q$-Pochhammers:
we use solid (resp. dashed) lines to indicate endpoints whose differences contribute to the $q$-Pochhammers in the denominator (resp. numerator).
For example, the special solid line on the top left corner connecting $0$ and $t_{11}$ corresponds to $(t_{11} - 0)^{-1}$ in the measure.
\begin{center}
  \includegraphics{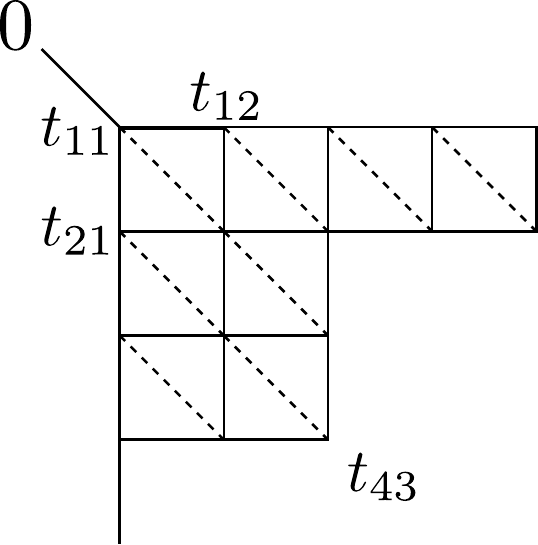}
\end{center}

The proof is about transformation by $\rho_{n, k}$ from measure $\mu_{\Theta, q}$ to $\mu_{\Lambda, q}$ where $\Lambda / \Theta = (n, k)$.
Without loss of generality assume $n > k$.
Intuitively speaking, after cancellations of $q$-Pochhammers that are not affected during this transformation, it suffices to show that $\rho_{n, k}$ has the following illustrated effect:
\begin{center}
  \includegraphics[scale=.7]{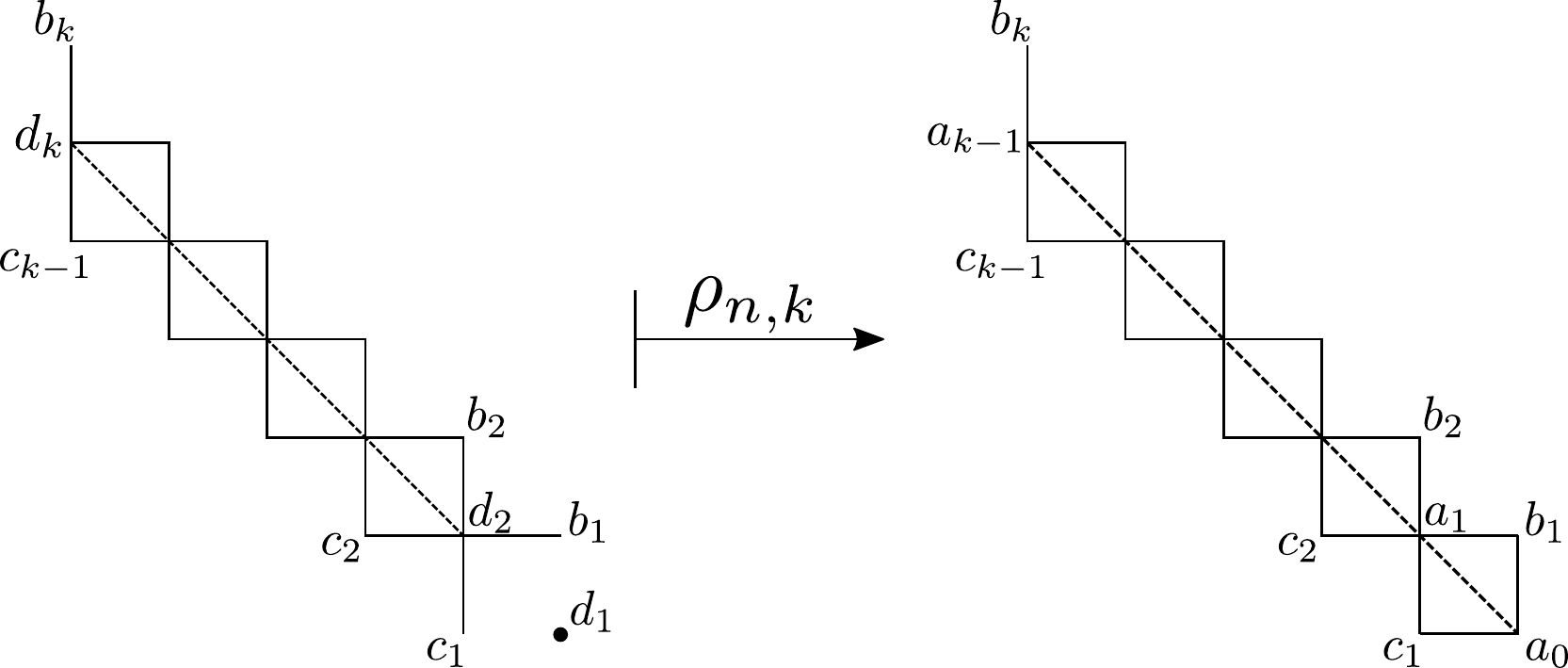}
\end{center}
Where the $a_i$'s, $b_i$'s, $c_i$'s and $d_i$'s are aliases of $t_{\ell, j}$'s on the tridiagonal area $\{(\ell, j): n - k - 1 \le \ell - j \le n - k + 1\}$ and the precise definition can be found in the proof.

Now let us turn to the complete proof. 
It may be compared to that of \cite[Theorem 3.2]{oconnell-seppalainen-zygouras14}.
\begin{proof}[Proof of Theorem \ref{t:lmpush}]
  We want to show
  \begin{align*}
    \sum_{(w_{i, j})_{(i, j) \in \Lambda}} \left(\prod_{(i, j) \in \Lambda} f_{\qGeom(\hat\alpha_i \alpha_j)} (w_{i, j})\right) \prob(T_\Lambda A = t) = \mu_{q, \Lambda} (t)
  \end{align*}

  First we can remove the matching product $\prod (\hat\alpha_i \alpha_j; q)_\infty$ from both sides.

  By Lemma \ref{l:wt} and Theorem \ref{t:qlocalmoves} we have that almost surely 
  \begin{align*}
    \sum_{i = 1 : \Lambda'_j} w_{i, j} = 
    \begin{cases}
      (T_\Lambda A)_{\Lambda'_1, 1} & j = 1 \\
      \sum_{k = 1 : j \wedge \Lambda'_j} (T_\Lambda A)_{\Lambda'_j - k + 1, j - k + 1} - \sum_{k = 1 : (j - 1) \wedge \Lambda_j} (T_\Lambda A)_{\Lambda_j - k + 1, j - k} & j > 1
    \end{cases}\\
    \sum_{j = 1 : \Lambda_i} w_{i, j} =
    \begin{cases}
      (T_\Lambda A)_{1, \Lambda_1}  & i = 1 \\
    \sum_{k = 1 : i \wedge \Lambda_i} (T_\Lambda A)_{i - k + 1, \Lambda_i - k + 1} - \sum_{k = 1 : (i - 1) \wedge \Lambda_i} (T_\Lambda A)_{i - k, \Lambda_i - k + 1} & i > 1
    \end{cases}
  \end{align*}
   Therefore the power of $\hat \alpha_i$ and $\alpha_j$ on both sides match, which we can also remove from the identity, leaving it sufficient to show
  \begin{align}
    \sum_{(w_{i, j})_{(i, j) \in \Lambda}} \left(\prod_{(i, j) \in \Lambda} (w_{i, j})_q^{-1}\right) \prob(T_\Lambda A = t) = M_\Lambda (t) \label{eq:m1}
  \end{align}
  where
  \begin{align*}
    M_\Lambda (t) = (t_{1 1})_q^{-1} {\prod_{(i, j) \in \Lambda: (i - 1, j - 1) \in \Lambda} (t_{i j} - t_{i - 1, j - 1})_q \over \prod_{(i, j) \in \Lambda: (i, j - 1) \in \Lambda} (t_{i j} - t_{i, j - 1})_q \prod_{(i, j) \in \Lambda: (i - 1, j) \in \Lambda} (t_{i j} - t_{i - 1, j})_q}.
  \end{align*}

  Once again we show this by an induction argument.
  When $\Lambda = (1)$ has just one coordinate, the left hand side of \eqref{eq:m1} becomes
  \begin{align*}
    \sum_{w_{11}} (w_{11})_q^{-1} \prob(l_{11} A = t) = (t_{11})_q^{-1} = M_\Lambda(t)
  \end{align*}
  is the right hand side of \eqref{eq:m1}.

Let $\Theta$ be a Young diagram such that $\Theta \nearrow \Lambda$.
Let $(n, k) = \Lambda / \Theta$.
Since $T_\Lambda = \rho_{n, k} \circ T_\theta$, we can rewrite the left hand side of \eqref{eq:m1} as
\begin{align*}
  \sum_{w_{n, k}} \sum_{(w_{i, j})_{(i, j) \in \Theta}} (w_{n, k})_q^{-1} \prod_{(i, j) \in \Theta} (w_{i, j})_q^{-1} \sum_{t': t'_{n, k} = w_{n, k}} \prob(T_\Theta A = t') \prob(\rho_{n, k} t' = t) \\
  = \sum_{t'} (t'_{n, k})_q^{-1} M_\Theta (t') \prob(\rho_{n, k} t' = t).
\end{align*}
where the last identity comes from the induction assumption.

So it suffices to show
\begin{align}
  \sum_{t'} (t'_{n, k})_q^{-1} {M_\Theta (t') \over M_\Lambda (t)} \prob(\rho_{n, k} t' = t) = 1. \label{eq:m4}
\end{align}

We assume $n > k$, as what follows can be adapted to the case $n < k$ due to the symmetry.
The proof when $n = k$ is similar with very minor changes.
For example, the right hand side of \eqref{eq:m4.5} will be the same except $(a_{k - 1} - b_k)_q$ and $(d_k - b_k)_q^{-1}$ are replaced by $(a_{k - 1})_q$ and $(d_k)_q^{-1}$ respectively due to the involvement of $(t'_{11})_q^{-1}$ and $(t_{11})_q^{-1}$.

We return to the proof where we assume $n > k$. 

When $k = 1$,
\begin{align*}
  (t'_{n, k})_q^{-1} {M_\Theta (t') \over M_\Lambda (t)} = (t'_{n, k})_q^{-1} (t_{n, k} - t_{n - 1, k})_q
\end{align*}
and due to the $q$-local moves on the boundary
\begin{align*}
  \prob(\rho_{n, k} t' = t) = \ind_{t_{n, k} = t'_{n, k} + t'_{n - 1, k}, t_{i, j} = t'_{i, j} \forall (i, j) \neq (n, k)}
\end{align*}
and we arrive at \eqref{eq:m4}.

When $k > 1$, since $\rho_{n, k}$ only changes the coordinates $B = \{(n, k), (n - 1, k - 1), \dots, (1, n - k + 1)\}$ of the matrix, the sum in \eqref{eq:m4} is over $(t'_{n - i + 1, k - i + 1})_{i = 1 : k}$.

By the structure of the products in $M_\Theta$ and $M_\Lambda$, we see that all the products outside of the diagonal strip near $(i - j) = n - k$ are cancelled out in $M_{k - 1} (t') / M_k (t)$. 
More precisely, when $k > 1$, by denoting
\begin{align*}
  a_i &= t_{n - i, k - i}, \qquad i = 0 : k - 1 \\
  b_i &= t_{n - i, k - i + 1} = t'_{n - i, k - i + 1}, \qquad i = 1 : k \\
  c_i &= t_{n - i + 1, k - i} = t'_{n - i + 1, k - i}, \qquad i = 1 : k - 1 \\
  d_i &= t'_{n - i + 1, k - i + 1}, \qquad i = 1 : k
\end{align*}
we have

\begin{align}
  (t'_{n, k})_q^{-1}{M_\Theta (t') \over M_\Lambda(t)} = (d_1)_q^{-1} (b_1 - d_2)_q^{-1} (c_1 - d_2)_q^{-1} (a_{k - 1} - b_k)_q (d_k - b_k)_q^{-1} {\prod_{i = 3 : k} h(d_{i + 1}, b_i, c_i, d_i) \over \prod_{i = 1 : k - 1} h(a_i, b_i, c_i, a_{i - 1})}
  \label{eq:m4.5}
\end{align}
where
\begin{align*}
  h(a', b', c', d') = (d' - a')_q (b' - a')_q^{-1} (c' - a')_q^{-1} (d' - c')_q^{-1} (d' - b')_q^{-1}
\end{align*}

It is time to calculate $\prob(\rho_{n, k} (t') = t)$.
This is the probability of mapping $(d_1, d_2, \dots, d_k)$ to $(a_0, a_1, \dots, a_{k - 1})$.

By the definition of the $q$-local moves (also see \eqref{eq:t1} and \eqref{eq:t2}) we have
\begin{align}
  a_0 &= b_1 + c_1 - d_2 + d_1 - X_1 \label{eq:m5}\\
  a_i &= b_{i + 1} + c_{i + 1} - d_{i + 2} + X_i - X_{i + 1}, \qquad i = 1 : k - 2 \label{eq:m6}\\
  a_{k - 1} &= b_k + X_{k - 1} \label{eq:m7}
\end{align}
where
\begin{align*}
  X_1 &\sim \qHyp(c_1 - d_2, \infty, b_1 - d_2) \\
  X_i &\sim \qHyp(c_{i} - d_{i + 1}, d_{i} - c_{i} ,b_{i} - d_{i + 1}), \qquad i = 2 : k - 1 
\end{align*}

By denoting $d_1 = X_0$, we can pin down the $X_i$'s in terms of the other variables.
\begin{align*}
  X_i &= \sum_{j = i : k - 1} a_j - \sum_{j = i + 1 : k} b_j - \sum_{j = i + 1 : k - 1} c_j + \sum_{j = i + 2 : k} d_j, \qquad i = 0 : k - 1
\end{align*}

Therefore
\begin{align*}
  \prob(&\rho_{n, k} (t') = t) \\
  &= f_{\qHyp(c_1 - d_2, \infty, b_1 - d_2)} (X_1) \prod_{i = 2 : k - 1} f_{\qHyp(c_{i} - d_{i + 1}, d_{i} - c_{i} ,b_{i} - d_{i + 1})} (X_i)
\end{align*}

Since we have the pmf's 
\begin{align*}
  &f_{\qHyp(c_1 - d_2, \infty, b_1 - d_2)} (X_1) = q^{(c_1 - d_2 - X_1) (b_1 - d_2 - X_1)} {\boxed{(c_1 - d_2)_q (b_1 - d_2)_q} \over (X_1)_q (c_1 - d_2 - X_1)_q (b_1 - d_2 - X_1)_q}\\
  &f_{\qHyp(c_{i} - d_{i + 1}, d_{i} - c_{i} ,b_{i} - d_{i + 1})} (X_i) = q^{(c_{i} - d_{i + 1} - X_i) (b_{i} - d_{i + 1} - X_i)} \\ 
  &\;\;\;\;\;\;\;\;\;\;\;\;\;\; \times \boxed{(c_{i} - d_{i + 1})_q (d_{i} - c_{i})_q (b_{i} - d_{i + 1})_q (d_{i} - b_{i})_q (d_{i} - d_{i + 1})_q^{-1}} \\
 &\;\;\;\;\;\;\;\;\;\;\;\;\;\; \times (X_i)_q^{-1} (c_{i} - d_{i + 1} - X_i)_q^{-1} (b_{i} - d_{i + 1} - X_i)_q^{-1} (d_i + d_{i + 1} - b_i - c_i + X_i)_q^{-1}\\
&\qquad\qquad\qquad\qquad\qquad\qquad\qquad\qquad\qquad\qquad\qquad\qquad\qquad\qquad i = 2 : k - 1
\end{align*}

The framed terms cancel out the $(c_1 - d_2)_q^{-1} (b_1 - d_2)_q^{-1} \prod_i h(d_{i + 1}, b_i, c_i, d_i)$ term in \eqref{eq:m4.5}. 

By shifting the terms not involving $d_{2 : k - 1}$ or $X_{0 : k - 1}$ from the left hand side of \eqref{eq:m4} to the right hand side we are left with showing

\begin{equation}
\begin{aligned}
  \sum_{d_{2 : k - 1}} &\prod_{i = 1 : k - 1} q^{(c_i - d_{i + 1} - X_i) (b_i - d_{i + 1} - X_i)} \prod_{i = 0 : k - 1} (X_i)_q^{-1} \\
  &\times \prod_{i = 1 : k - 1} \left((c_i - d_{i + 1} - X_i)_q^{-1} (b_i - d_{i + 1} - X_i)_q^{-1}\right) \\
  &\times \prod_{i = 2 : k - 1} (d_i + d_{i + 1} - b_i - c_i + X_i)_q^{-1} \times (d_k - b_k)_q^{-1}\\
  &= (a_{k - 1} - b_k)_q^{-1} \prod_{i = 1 : k - 1} h(a_i, b_i, c_i, a_{i - 1})
\end{aligned}
  \label{eq:m8}
\end{equation}

By the relation between $X_i$ and $X_{i + 1}$ in \eqref{eq:m5}\eqref{eq:m6} as well as the explicit form of $X_{k - 1}$ in \eqref{eq:m7}, we can write
\begin{align*}
  (X_i)_q^{-1} &= (X_{i + 1} + a_i - b_{i + 1} - c_{i + 1} + d_{i + 2})_q^{-1} \qquad i = 0 : k - 2 \\
  (d_i + d_{i + 1} - b_i - c_i + X_i)_q^{-1} &= (X_{i - 1} + d_i - a_{i - 1})_q^{-1}, \qquad i = 2 : k - 1\\
  (d_k - b_k)_q^{-1} &= (X_{k - 1} + d_k - a_{k - 1})_q^{-1}
\end{align*}
plugging these back into the left hand side of \eqref{eq:m8}, it becomes

\begin{align*}
  (X_{k - 1})_q^{-1} \sum_{d_k} f_k(d_k) \sum_{d_{k - 1}} f_{k - 1}(d_{k - 1}) \sum_{d_{k - 2}} f_{k - 2}(d_{k - 2}) \dots \sum_{d_3} f_3(d_3) \sum_{d_2} f_2(d_2)
\end{align*}

where

\begin{align*}
  f_{i } (d_{i }) = &q^{(c_{i - 1} - d_{i } - X_{i - 1}) (b_{i - 1} - d_{i } - X_{i - 1})} (X_{i - 1} + a_{i - 2} - b_{i - 1} - c_{i - 1} + d_{i })_q^{-1} \\
  &\times (c_{i - 1} - d_{i} - X_{i - 1})_q^{-1} (b_{i - 1} - d_{i} - X_{i - 1})_q^{-1} (X_{i - 1} + d_{i } - a_{i - 1})_q^{-1}
\end{align*}

Note that $f_i (d_i)$ depends on $X_{i - 1}$, which in turn depends on $d_{i + 1 : k - 1}$.

However the sum remove the dependencies of all the $d$'s.
More precisely, starting from the innermost sum $\sum_{d_2} f_2 (d_2)$, 
by applying \eqref{eq:qhyp} with $(m_1, m_2, k, s) := (c_{1} - a_{1}, a_{0} - c_{1} , b_{1} - a_{1}, d_2 + X_{1} - a_{1})$ , we have
\begin{align*}
  \sum_{d_2} f_2 (d_2) = h(a_1, b_1, c_1, a_0)
\end{align*}
which has no dependency on the $d$'s.

So we can recursively apply \eqref{eq:qhyp} with $(m_1, m_2, k, s) := (c_{i - 1} - a_{i - 1}, a_{i - 2} - c_{i - 1} , b_{i - 1} - a_{i - 1}, d_i + X_{i - 1} - a_{i - 1})$ , to obtain
\begin{align*}
  \sum_{d_i} f_i (d_i) = h(a_{i - 1}, b_{i - 1}, c_{i - 1}, a_{i - 2}), \qquad i = 2 : k
\end{align*}

This leaves us with only $(a_{k - 1} - b_k)_q^{-1}$ on the right hand side of \eqref{eq:m8}. 
Since on the left hand side $X_{k - 1} = a_{k - 1} - b_k$ we are done.
\end{proof}

\subsection{Joint distribution of $q$-polymer and polynuclear growth models}
By applying Theorem \ref{t:qlocalmoves} and Theorem \ref{t:lmpush} we obtain the joint distribution of $q$-polymer, Corollary \ref{c:qpdist}.

When $\Lambda = (p, p - 1, p - 2, \dots, 1)$ is a staircase Young diagram, the $q$-local moves defines a $q$-version of the multilayer polynuclear growth model.

Recall in \cite{johansson03}, the PNG model was concerned with the height function $h^j_m(k)$ at time $m$, position $k$ and level $j$, where all of time, space and level are discrete.
The level starts from $0$ onwards, where we call $h^0_m(k)$ the top level height function and abbreviate it as $h_m(k)$.
The height function is $0$ outside of a cone: $h^j_m(k) = 0$ for $|k| \ge m - 2 j$.
The initial condition is $h^j_0(0) = 0$ for all $j$, and the height functions grow as they are fed by the droplets over the time.
We denote the droplet at time $m$ and position $k$ by $d_m(k)$, which is also zero outside of the cone $|k| < m$ or when $k$ and $m$ have same parity.
Later we will see that these are droplets for the top level, and there will be droplets for the non top levels as well.
Hence it is useful to denote $d^0_m(k) := d_m(k)$ and use $d^j_m(k)$ as notations for droplets at level $j$ in general.
The PNG model evolves as follows:
\begin{enumerate}
\item At time $1$, the droplet at position $0$ forms the height at the same location: $h_1(0) = d_1(0)$.
\item At time $2$, the height expands horizontally by $1$ to both directions ($h_2(-1) = h_2(0) = h_2(1) = h_1(0)$), and droplets at the new positions ($\pm1$) adds to the height function ($h_2(-1) := h_2(-1) + d_2(-1)$, $h_2(1) := h_2(1) + d_2(1)$). So the net effect is:
  \begin{align*}
    h_2(-1) = h_1(0) + d_2(-1), \qquad h_2(0) = h_1(0), \qquad h_2(1) = h_1(0) + d_2(1).
  \end{align*}
\item At time $3$, the peak heights (namely the ones at position $\pm1$) further expands horizontally by $1$ to both directions, and at positions $-2, -1, 1$ and $2$ the same event as at time $2$ happens:
  \begin{align*}
    h_3(-2) &= h_2(-1) + d_3(-2), \qquad h_3(2) = h_2(1) + d_3(2),\\
    h_3(-1) &= h_2(-1), \qquad h_3(1) = h_2(1).
  \end{align*}
  However, at position $0$, the expansions from positions $-1$ and $1$ collide, in which case the maximum of the colliding heights remains on the top level, whose sum with the droplet $d_3(0)$ forms the new height, and the minimum becomes the droplet for the first level and forms the height at the first level:
  \begin{align*}
    h_3(0) = (h_2(-1) \vee h_2(1)) + d_3(0), \qquad h^1_3(0) = d^1_3(0) = h_2(-1) \wedge h_2(1).
  \end{align*}
\item At any time, starting from the top level, the height function at each level expands both ways. And at any collision, the sum of the maximum of the colliding heights and the droplet becomes the new height, and the minimum becomes the droplet for the next level at the same position:
  \begin{equation}
  \begin{aligned}
    h^j_m(k) &= (h^j_{m - 1} (k - 1) \vee h^j_{m - 1} (k + 1)) + d^j_m(k)\\
    d^{j + 1}_m(k) &= h^j_m(k - 1) \wedge h^j_m(k + 1). 
  \end{aligned}
    \label{eq:png}
  \end{equation}
\end{enumerate}

Clearly given that all droplets are sampled independently, the height functions are a Markov process because their values at time $n$ only depend on their values at time $n - 1$ and the droplets at time $n$.

It is known that the PNG model observes the same dynamics as the RSK algorithm acting on a staircase tableau.
More specifically, one may identify the top-level droplets for PNG with the input data for RSK:
\begin{align}
  d_m(k) = w_{\lfloor {m - k \over 2} \rfloor, \lceil{m + k \over 2}\rceil } \label{eq:rskpng1}
\end{align}
Let $q = 0$, then by identifying
\begin{align}
  h^j_m(k) &= t_{\lfloor {m - k \over 2} \rfloor - j, \lceil{m + k \over 2}\rceil - j} \label{eq:rskpng2}
\end{align}
where the $t_{n\ell}$'s are the output of local moves taking the staircase tableau $(w_{ij})_{i + j \le m + 1}$, one may recover the dynamics of the PNG model \eqref{eq:png} from the dynamics of the local moves.

Using the same correspondence, the gPNG model was defined using the gRSK dynamics, as per \cite{nguyen-zygouras16}.

Similarly we may define a $q$-version of the PNG model using the same correspondence \eqref{eq:rskpng1}\eqref{eq:rskpng2} for $0 < q < 1$.
With the same reasoning, one can say that the $q$PNG height functions are a Markov process.
The dynamics is a bit more hairy than the usual PNG but a simple rewriting of the $q$RSK algorithm. 
Here we show the dynamics of the top level height function:
\begin{align*}
  h_m(k) &= h_{m - 1} (k - 1) + h_{m - 1} (k + 1) - h_{m - 1} (k) \\
  &\qquad - \qHyp(h_{m - 1}(k - 1) - h_{m - 1}(k), \infty, h_{m - 1}(k + 1) - h_{m - 1}(k)) + d_m(k).
\end{align*}
It can be seen from this formula that, similar the usual PNG model, the height function $h_m(k)$ is a function of the heights at neighbouring positions at the previous time $h_m(k - 1), h_m(k), h_m(k + 1)$ and the droplet $d_m(k)$.

In \cite{johansson03} the PNG model was used to show that asymptotically the partition functions of DLPP at the same time are the Airy process.

Here by applying the $q$-local moves on the staircase Young diagram and use Theorem \ref{t:lmpush} and Theorem \ref{t:qlocalmoves}, we obtain a $q$-version and the joint distribution of partition functions of the $q$-polymer at a fixed time in Corollary \ref{c:qpng} in Section \ref{s:intro}.

Our result on joint distributions of polymers, Corollary \ref{c:qpdist} and \ref{c:qpng} are $q$-versions of Theorem 2.8 and 3.5 in \cite{nguyen-zygouras16} respectively. To obtain anything more, such as the $q$-version of the two-point Laplace transform in Theorem 2.12 in that paper or the central limit of one point partition function in Theorem 1 in \cite{borodin-corwin-remenik13}, a natural question arises whether one can obtain a $q$-Whittaker version of Corollary 1.8 (writing one-point Laplace transform as a Fredholm determinant) in \cite{borodin-corwin-remenik13}, which is the main tool to show the two results.

\subsection{Measures on the matrix}

In this section we restrict to $\Lambda = [n] \times [m]$.  As a straightfoward application of Theorem \ref{t:qlocalmoves} one can show the following result, from which Corollary \ref{c:qwmeasure} follows:
\begin{prp}
  The distribution of the marginal variable $\lambda = (t_{n, m}, t_{n - 1, m - 1}, \dots, t_{(n - m)^+ + 1, (m - n)^+ + 1})$ of $(t_{ij})_{(i, j) \in \Lambda}$ is the $q$-Whittaker measure
  \begin{align*}
    \sum_{t_{i j}: i - j \neq n - m, 1 \le i \le n, 1 \le j \le m} \mu_q(t) = \mu_{q\text{-Whittaker}} (\lambda) = P_\lambda(\alpha) Q_\lambda(\hat\alpha) \prod_{i, j} (\hat\alpha_i \alpha_j; q)_\infty
  \end{align*}
\end{prp}

Let $L$ be a measure on $\real^{n \times m}$ defined as follows
\begin{align*}
  L(x) &= \exp(- e^{- x_{1 1}}) \prod_{i = 1 : n} \prod_{j = 2 : m} \exp(- e^{x_{i, j - 1} - x_{i j}}) \prod_{i = 2 : n} \prod_{j = 1 : m} \exp(- e^{x_{i - 1, j} - x_{i, j}}) \\
  &\;\;\;\;\;\;\;\;\;\;\;\;\times \exp( - \sum  \theta_i s_i) \exp(- \sum \hat\theta_i \hat s_i ) \prod_{i = 1 : n} \prod_{j = 1 : m} \Gamma(\hat\theta_i + \theta_j)^{-1}
\end{align*}
where
\begin{align*}
  s_1 &= t_{n, 1} \\
  s_i &= \sum_{j = 1 : n \wedge i} t_{n - j + 1, i - j + 1} - \sum_{j = 1 : n \wedge (i - 1)} t_{n - j + 1, i - j}, \qquad i = 2 : m\\
  \hat s_1 &= t_{1, m} \\
  \hat s_i &= \sum_{j = 1 : m \wedge i} t_{i - j + 1, m - j + 1} - \sum_{j = 1 : m \wedge (i - 1)} t_{i - j, m - j + 1}, \qquad i = 2 : n.
\end{align*}
This measure was introduced in \cite{oconnell-seppalainen-zygouras14} as the push-forward measure of local moves acting on a matrix with log-Gamma weights.

The next proposition demonstrates the classical limit from $\mu_q$ to $L$.
\begin{prp}
  Let 
  \begin{align*}
    q &= e^{- \epsilon} \\
    t_{i j} &= (i + j - 1) m(\epsilon) + \epsilon^{- 1} x_{i j} \\
    \hat\alpha_i &= e^{- \epsilon \hat\theta_i} \\
    \alpha_j &= e^{- \epsilon \theta_j}
  \end{align*}
  Then
  \begin{align*}
    \lim_{\epsilon \downarrow 0} \epsilon^{- m n} \mu_q(t) = L(x).
  \end{align*}
\end{prp}

\begin{proof}
  Quite straightforward by plugging in Items 1 and 2 in Lemma \ref{l:qclassicallimits}.
\end{proof}

\Addresses
\end{document}